\documentclass{amsart}
\usepackage{tikz,tikz-cd}
\usepackage{url}
\usepackage{graphicx}
\usepackage{color}
\usepackage{times}
\usepackage{cite}
\usepackage[colorlinks]{hyperref}
\usepackage{pdfsync}
\usepackage{enumerate,latexsym}
\usepackage{amsmath,amssymb}
\usepackage{amsthm}
\usepackage{verbatim}
\newtheorem{thm}{Theorem}[section]
\newtheorem*{theorem*}{Theorem}
\newtheorem*{acknowledgement*}{Acknowledgement}
\newtheorem{cor}[thm]{Corollary}
\newtheorem{claim}[thm]{Claim}

\newtheorem{lem}[thm]{Lemma}
\newtheorem{prop}[thm]{Proposition}
\theoremstyle{definition}
\newtheorem{defn}[thm]{Definition}
\theoremstyle{remark}
\newtheorem{rem}[thm]{Remark}

\numberwithin{equation}{section}

\newcommand{\func}[1]{\ensuremath{\mathop{\mathrm{#1}}} }

\newcommand{\spt}[0]{\func{spt}}

\newcommand{\e}{\varepsilon}

\author{Huy The Nguyen}
\address{School of Mathematical Sciences\\
	Queen Mary University of London\\
	Mile End Road\\
	London E1 4NS}
\email{h.nguyen@qmul.ac.uk}
\author{Shengwen Wang} \address{Mathematics Institute\\
	University of Warwick\\
	Gibbet Hill Road\\
	Coventry CV4 7AL}\email{shengwen.wang@warwick.ac.uk}
\begin{document}
\title[Parabolic Allen--Cahn]{Second order estimates for transition layers and a curvature estimate for the parabolic Allen--Cahn}
\maketitle
\begin{abstract}
The parabolic Allen--Cahn equation is a semilinear partial differential equation that is closely linked to the mean curvature flow by a singular perturbation. Motivated by the work of Wang--Wei \cite{Wang2019a} and Chodosh--Mantoulidis \cite{Chodosh2020} in the elliptic setting, we initiate the corresponding regularity theory for parabolic Allen--Cahn flows. In particular, we establish an improved convergence property of parabolic Allen--Cahn flows to the mean curvature flow: if the phase-transition level sets converge in $C^2$, then they converge in $C^{2,\theta}$ as well. As an application, we obtain a curvature estimate for the parabolic Allen--Cahn equation, which can be seen as a diffused version of Brakke's \cite{brakke2015motion} and White's \cite{white2005local} regularity theorems for mean curvature flow.
\end{abstract}

\section{Introduction}

The parabolic Allen--Cahn equation
	\begin{align}\label{ACF}
	\frac{\partial}{\partial t}u=\Delta u-W'(u)
	\end{align}
is an evolution equation that models the reaction-diffusion dynamics of phase transition. It is the gradient flow of the Allen--Cahn phase separation energy
\begin{equation*}
E(u)=\int\frac{1}{2}|\nabla u|^2+W(u)
\end{equation*}
 where $W(u):\mathbb R\rightarrow\mathbb R$ is a double-well shaped potential function.

Geometrically, the Allen--Cahn equation has a close relationship with mean curvature flow through its singularly perturbed version
 	\begin{align}\label{AFCEpsilon}
	\frac{\partial}{\partial t}u_\e=\Delta u_\e-\frac{W'(u_\e)}{\e^2}.
	\end{align}
The two equations are related by the parabolic scaling $u_\e(x,t)=u(\frac{x}{\e},\frac{t}{\e^2})$. In particular, equation is not scale invariant but $u_\e$ satisfies an $\e$-equation of the same form but with a different parameter. It was shown by Ilmanen \cite{Ilmanen1993} as the parameter $\e\rightarrow0$, the energy measure
	\begin{align*}
	d\mu_\e(u)=\left[\frac{\e|\nabla u_\e|^2}{2}+\frac{W(u_\e)}{\e}\right] dx
	\end{align*} of the $\e$-solution converges in the sense of varifolds to Brakke's weak mean curvature flow. Moreover, the limit Brakke flow has integer multiplicity a.e. by Tonegawa \cite{tonegawa2003integrality}. Hence the parabolic Allen--Cahn is a model for the flow of mean curvature flow through singularities. In particular, note the equation is a subcritical semilinear equation and hence does not form singularities as $t\rightarrow \infty$. This property makes the flow an appealing candidate for weak mean curvature flow.

For geometric applications, it is necessary to obtain higher regularity for the convergence. In the elliptic setting, Caffarelli--Cordoba \cite{caffarelli2006phase} showed the transition layers of phase transitions have uniform $C^{1,\theta}$ regularity (depending only upon the Lipschitz norm and independent of $\e$). Wang--Wei \cite{Wang2019a, Wang2019} proved stable transition layers converges in a stronger $C^{2,\theta}$ sense to the limit minimal surfaces in dimensions $n\leq10$. Using an improvement of the convergence in dimension $3$, Chodosh-Mantoulidis \cite{Chodosh2020} proved that the min-max minimal surfaces obtained from the Allen--Cahn construction by Guaraco \cite{Guaraco2018} in a generic 3-manifold has multiplicity $1$ and expected index. This gives an alternative proof of Yau's conjecture of existence of infinitely many minimal surfaces. These results differ from the methods used in \cite{Ilmanen1993}, \cite{tonegawa2003integrality}. They do not use geometric measure theoretic techniques, but instead uses an infinite dimensional Lyapunov--Schmidt reduction developed in \cite{Pino2011} and \cite{Pino2013}.

Motivated by the work of \cite{Wang2019a, Wang2019} and \cite{Chodosh2020} in the elliptic setting, we initiate the corresponding regularity theory for parabolic Allen--Cahn. In particular, for low entropy parabolic Allen--Cahn flow, we have an improved convergence of their transition layers to mean curvature flow. The motivation in the elliptic setting was minimal surfaces, in particular a proof Yau's conjecture \cite{Chodosh2020}, in the parabolic setting, the corresponding problem is the multiplicity $1$ conjecture for mean curvature flow by Ilmanen. It is expected the parabolic Allen--Cahn equation and its improved convergence properties will have applications in understanding mean curvature flow and its singularities. The key idea in this paper is a parabolic analogue of the Lyapunov--Schmidt reduction, \eqref{eqn_best_approx}. In the parabolic case, this was first used in \cite{Pino2018} and \cite{Pino2018a}.

Our main result is the following theorem, which improves the regularity of level sets
\begin{thm}\label{ImprovementRegularity}
For any $\delta>0$, let $u_\e$ be a sequence of solutions to \eqref{AFCEpsilon} with $\e\rightarrow0$ in a space-time open set $B_2(0)\times[-2,2]\subset\mathbb R^{n+1}\times\mathbb R$ such that $\{u_\e(\cdot, t)=0\}\cap B_1(0)\neq\emptyset$ for all $t\in[-2,2]$. Furthermore let us assume that the entropy (see Definition \ref{DefinitionEntropy}) satisfies $\lambda(d\mu_\e)<2\alpha-\delta$ where $d\mu_\e$ is the energy measure of $u_\e$ (see \eqref{EnergyMeasureExpression}), and the enhanced second fundamental form is uniformly bounded by $\mathcal A(u_\e)\leq C$ (see section 2 for the definition), where $C$ is a uniform constant independent of $\e$.

Then the nodal sets $\{u_\e=0\}$ converge in the parabolic $C^{2,\theta}$ sense to a smooth mean curvature flow in $B_1(0)\times[-1,1]\subset\mathbb R^{n+1}\times\mathbb R$, for any $\theta\in(0,1)$.

In particular, the spatial $C^\theta$ H\"older norm of the second fundamental form of the nodal sets and the $C^{1,\frac{\theta}{2}}$ norm of the time derivatives are uniformly bounded on compact subsets.
\end{thm}

\begin{rem}
This theorem is the parabolic analogue of Theorem 1 in \cite{Wang2019}, where the stability condition in the elliptic setting is replaced by a low entropy condition. We note the low entropy condition ensures we only have one transition layer which substantially simplifies the analysis. In particular, we are not required to model interactions between separate layers and hence do not need Toda systems.

The condition in the theorem implicitly implies the limit mean curvature flow is smooth, because $C^2$ bounds imply $C^{1,\theta}$ convergence of the transition layers, and standard regularity theory for quasilinear parabolic partial differential equations allows us to bootstrap $C^{1,\theta}$ bounds to $C^\infty$ smoothness of the limit flow.
\end{rem}

\begin{thm}\label{Graphical}
For any $\delta>0$, suppose $u_\e$ is a solution of \eqref{AFCEpsilon} with $\e\rightarrow0$ in a space-time open set $B_2(0)\times[-2,2]\subset\mathbb R^{n+1}\times\mathbb R$ such that the entropy $\lambda(d\mu_\e)<2\alpha-\delta$, the nodal sets $\Gamma_{\e,t}=\{u_\e(x,t)=0\}$ of $u_\e$ in $B_2(0)\times[-2,2]$ can be represented by a Lipschitz graph over the limit mean curvature flow $\Sigma_t$ as
\begin{equation*}
\Gamma_{\e,t}=\mathrm{Graph}_{\Sigma_t}f_{\e,t}
\end{equation*}
with $f_{\e,t}$ having uniformly bounded Lipschitz norms in $B_2(0)\times[-2,2]$.

Then the same conclusion as in Theorem \ref{ImprovementRegularity} holds.
\end{thm}

\begin{rem}
This result is similar to Corollary 1.2 in \cite{Wang2019}. In the elliptic case, if the function is rescaled to have bounded curvature, the stability of $u$ guarantees that the blow-up limit is either a one-dimensional solution to the Allen--Cahn equation in $\mathbb R^{n+1}$ or a flat hyperplane. In our case we cannot guarantee the limit is flat, hence we show the nodal sets are graphical over the limiting mean curvature flow instead of over a hyperplane as in \cite{Wang2019}. We will use this theorem in a subsequent paper to prove a Brakke-type regularity theorem for the parabolic Allen--Cahn equation \cite{Nguyen2020}.
\end{rem}

\begin{cor}\label{CurvatureEstimates}
For any $\delta>0$ there exists a $C_0>0,\e_0>0$ with the following consequence: Let $u_{\e}$ be a solution of \eqref{EAC} defined on $\mathbb R^2\times(-r^2,r^2)$, $u_{\e}(0,0)=0$. Suppose that there exists uniform $R_0, W_0,\delta_0>0$ such that the following holds:
\begin{enumerate}
\item $u_\e$ represents phase transition in $B_{R_0}$ for every $\e\leq\e_0$ and time $t\in(-r^2,r^2)$ (see Definition \ref{RepresentsTransition});
\item $\int_{\mathbb R^2}\e\left(\Delta u_\e(x,t)-\frac{W'(u_\e(x,t))}{\e^2}\right)^2dx\leq W_0, \forall t\in(-r^2,r^2)$;
\item The entropy $\lambda(d\mu_\e)\leq2\alpha-\delta$. 
\end{enumerate}
Then we must have $\nabla u_\e(0,0)\neq0$ and that the enhanced second fundamental form satisfies
\begin{equation*}
\mathcal A(u_{\e}(0,0))\leq\frac{C_0}{r}
\end{equation*}
with $C_0$ independent of $\tilde\e$.

We also obtain an improvement in the regularity of convergence result of \cite{trumper2008relaxation}. We show the level sets of parabolic Allen--Cahn converges in $C^{2,\theta}$ to the curve shortening flow in $\mathbb R^2$ when entropy is below $2\alpha$.
\end{cor}

\begin{rem}
This result can be seen as a relaxation of the curvature estimates found in Brakke's paper \cite{brakke2015motion} and White's paper \cite{white2005local} for curve shortening flow. However, in higher dimensions, we required the rigidity of eternal solutions to the parabolic Allen--Cahn equation with low entropy (as stated in Theorem \ref{ParabolicRigidity2D}) to prove this result. In a subsequent paper \cite{Nguyen2020}, we prove such a rigidity theorem for any dimension $n$, with an entropy bound $\lambda(d\mu_\e)\leq(1+\tau_n)\alpha$ for some $\tau_n>0$ and without the Willmore-type bound in Item (2) above. This allows us to obtain corresponding curvature estimates of the form in Corollary \ref{CurvatureEstimates} under that entropy bound.

Notably, the entropy bound is sharp in dimension $2$. Specifically, for a limit curve shortening flow, the Grim Reaper provides a counter-example, with an entropy of $2$, which can be rescaled to have arbitrarily large curvature.
\end{rem}
This paper is organised as follows: in section \ref{Preliminaries} we standardise our notation and provide background material for the parabolic Allen--Cahn equations. We also prove a rigidity of the blow-up limit of parabolic Allen--Cahn solutions in spatial dimension 2 in section \ref{RigiditySection}, which will be a necessary ingredient in the proof of the curvature estimate Corollary \ref{CurvatureEstimates}. In section \ref{ApproximateSolution} and section \ref{DerivationEquation} we carry out the main estimates and prove the main theorems in Section \ref{ProofMainTheorem}. Finally in section \ref{ProofCurvature} we prove the curvature estimates of low entropy solutions, Corollary \ref{CurvatureEstimates}.

$\textbf{Acknowledgements.}$
H.T.Nguyen was supported by the EPSRC grant EP/S012907/1, S.W was supported by EPSRC grant EP/S012907/1 and EP/T019824/1.

\section{Preliminaries and notations}\label{Preliminaries}

\subsection{Preliminaries about Allen--Cahn and the explicit $1$-d heteroclinic solution}

To ensure consistency throughout the remainder of the paper, we will adopt a particular notation.

It is worth noting that solutions to the Allen--Cahn equation with parameter $\e$ are not invariant under standard parabolic rescaling. However, rescaling the solution does produce equations of the same form, albeit with a different parameter. We will say $u_\e:\mathbb R^{n+1}\times\mathbb R\rightarrow\mathbb R$ satisfies the $\e$-equation if
	\begin{align}\label{EAC}
	\frac{\partial}{\partial t}u_\e=\Delta u_\e-\frac{W'(u_\e)}{\e^2}
	\end{align}
and say $u:\mathbb R^{n+1}\times\mathbb R\rightarrow\mathbb R$ satisfies the $1$-equation if
	\begin{align}\label{AC}
	\frac{\partial}{\partial t}u=\Delta u-W'(u).
	\end{align}

Here $W(u)=\frac{1}{4}(1-u^2)^2$ is the standard double-well potential which will be used through out the rest of the paper. In particular, the potential has two global minima $\pm1$ that represents $2$ stable phases and the function $u$ describes continuous transition between these phases.

Under parabolic rescaling, it is not hard to see if $u_\e(x,t)$ satisfies equation \eqref{EAC}, then
	\begin{align}\label{ParabolicRescaling}
	u^\e(x,t)=u_\e(\e x,\e^2 t)
	\end{align}
satisfies equation \eqref{AC}. Throughout this paper, we will adopt the convention that $u_\e$ with a subscript $\e$ will satisfy the $\e$-equation \eqref{EAC} and $u^\e$ with a superscript $\e$ is obtained by a parabolic rescaling of $u_\e$ and satisfies equation \eqref{AC}.

A static solution to the Allen--Cahn equation \eqref{AC} is a function $u:\mathbb R^{n+1}\rightarrow\mathbb R$ that satisfies the elliptic equation
	\begin{align}\label{EllipticAC}
	\Delta u-W'(u)=0
	\end{align}
and represents an equilibrium state of phase transition in $\mathbb R^{n+1}$.

\begin{defn}\label{RepresentsTransition}
We call a solution $u:\mathbb R^{n+1}\times\mathbb R\rightarrow\mathbb R$ represents a phase transition at time $t$ in $\Omega\subset\mathbb R^{n+1}$ if $\{u(\cdot,t)=0\}\cap\Omega\neq\emptyset$.
\end{defn}
Two trivial solutions to the equation that do not represent phase transitions are
\begin{equation*}
u(x)=\pm1
\end{equation*} which are the equilibrium states when there is only one phase (either one of the two phases $\pm1$) in the whole region $\mathbb R^{n+1}$. In dimension $1$, one can find explicitly the next simplest solution which represents a phase transition by solving the ordinary differential equation
\begin{equation*}
g''(x)-W'(g(x))=0
\end{equation*} where $g:\mathbb R\rightarrow\mathbb R$.

And the explicit solution satisfying the asymptotics $\lim_{x\rightarrow\pm\infty} g(x)=\pm1$ and $g(0)=0$ when $W(u)=\frac{1}{4}(1-u^2)^2$ is
\begin{equation*}
g(x)=\tanh(x).
\end{equation*}

From this $1$-d solution, we also obtain heteroclinic solutions to \eqref{EAC} with any $\e$ parameter and in any dimensions
	\begin{align}\label{1d}
	g_\e(x,t)=g_\e(x_1,...,x_{n+1},t)=\tanh\left(\frac{x_{n+1}}{\e}\right).
	\end{align}

We denote the total energy of the $1$-d heteroclinic solution by
\begin{equation*}
\alpha=E(g)=\int g'^2(x) dx.
\end{equation*}

For a closed initial hypersurface $M_0=\partial E_0\subset\mathbb R^{n+1}$, we can choose a sequence of measures $\mu_{\e,0}$ as in \cite[1.4]{Ilmanen1993} so that $\mu_{\e,0}\rightarrow\alpha\mathcal H^n\lfloor M_0$ as $\e\rightarrow0$. Then by the main results in \cite{Ilmanen1993, tonegawa2003integrality}, if the energy of the sequence is locally uniformly bounded, then the energy measures
	\begin{align}\label{EnergyMeasureExpression}
	d\mu_{\e,t}=\left[\frac{\e|\nabla u_\e(x,t)|^2}{2}+\frac{W(u_\e(x,t))}{\e}\right] dx\rightarrow d\mu_t
	\end{align}
	 converge as $\e\rightarrow0$ to an integer $n$-rectifiable varifold for a.e. $t$ and $\{\mu_t\}_{t\geq0}$ is a mean curvature flow in the sense of Brakke. We note that the entropy assumption in our theorems implies a locally uniform bound for the energy measures.

Similar to Brakke's integral form of mean curvature flow, there is an $\e$ version of the integral form of the parabolic Allen--Cahn equation \cite[3.1]{Ilmanen1993},
	\begin{align}\label{EBrakke}
	\begin{split}
\frac{d}{dt}\int\phi d\mu_{\e,t}=-\int_{\mathbb R^n}\e\phi\left(\Delta u_\e-\frac{W'(u_\e)}{\e}\right)^2 dx-\delta V_{\e,t}(\nabla\phi)-\int\nu\otimes\nu:\nabla^2\phi d\xi_{\e,t}.
	\end{split}
	\end{align}
The measure
\begin{equation*}
d\xi_{\e}=\left[\frac{\e|\nabla u_\e|^2}{2}-\frac{W(u_\e)}{\e}\right]dx
\end{equation*}
 is called the discrepancy measure and it is shown to converge to $0$ in $L^1$ as $\e\rightarrow0$ in \cite{Ilmanen1993, soner1997ginzburg, tonegawa2003integrality}. Additionally, $\delta V_t^\e$ is the first variation of the corresponding varifold; see \cite{Ilmanen1993} for details on how to consider $u_\e(x,t)$ as a general moving varifold whose density is the energy density $d\mu_{\e,t}$.
 
Analogous to Huisken's monotonicity formula in mean curvature flow, Ilmanen in \cite{Ilmanen1993} discovered an almost monotonicity formula for the $\e$-parabolic Allen--Cahn equation \eqref{EAC}
	\begin{align}\label{Monotonicity}
	\begin{split}
&\frac{d}{dt}\int\Psi_{y,s}d\mu_\e(u_\e)\\
	&= -\int_{\mathbb R^n}\e\Psi_{y,s}\left(\Delta u_\e-\frac{W'(u_\e)}{\e^2}+\frac{\nabla u_\e\cdot\nabla\Psi_{y,s}}{\Psi_{y,s}}\right)^2dx\\
	&+\int_{\mathbb R^n}\frac{1}{2(s-t)}\Psi_{y,s}\left[\frac{\e|\nabla u_\e|^2}{2}-\frac{W(u_\e)}{\e}\right]dx\\
	&= -\int_{\mathbb R^n}\e\Psi_{y,s}\left(\Delta u_\e-\frac{W'(u_\e)}{\e^2}+\frac{\nabla u_\e\cdot\nabla\Psi_{y,s}}{\Psi_{y,s}}\right)^2dx+\int\frac{1}{2(s-t)}\Psi_{y,s} d\xi_{\e,t}
\end{split}
	\end{align}
where $\Psi_{y,s}(x,t)=\frac{1}{(4\pi (s-t))^{\frac{n}{2}}}e^{-\frac{|x-y|^2}{4(s-t)}}$ is the $n$-dimensional backward heat kernel centred at $y\in\mathbb R^{n+1}$ with scale $s\in\mathbb R^+$.

It is also computed in \cite[\S 4]{Ilmanen1993} that non-positivity of the discrepancy is preserved in time and thus $\int\Psi_{y,s}d\mu_\e(u_\e)$ is monotone for initial data with non-positive discrepancy.
\subsection{Entropy}
Motivated by Colding--Minicozzi's \cite{colding2012generic} entropy in mean curvature flow, we follow \cite{sun2018entropy} and introduce the entropy functional on the space of Radon measures.
\begin{defn}\label{DefinitionEntropy}
Given a Radon measure $d\mu$ on $\mathbb R^{n+1}$, its entropy $\lambda(d\mu)$ is defined by
	\begin{align}
	\lambda(d\mu)&=\sup_{s>0,y\in\mathbb R^{n+1}}\int\Psi_{y,s}(x,0) d\mu\\\nonumber
&=\sup_{s>0,y\in\mathbb R^{n+1}}\int\frac{1}{(4\pi s)^{\frac{n-1}{2}}}e^{-\frac{|x-y|^2}{4s}} d\mu.
	\end{align}
\end{defn}
We also notice, by an observation of Sun \cite{sun2018entropy}, if the entropy $\lambda$ is below $2\alpha-\delta$ for some $\delta>0$, then the limit mean curvature flow has unit density.

\subsection{Geometry of the Allen--Cahn level sets}

For a non-degenerate point $x\in\mathbb R^{n+1}$ where $|\nabla u|\neq0$, the normal vector of the level set is given by $\nu(x)=\frac{\nabla u}{|\nabla u|}$. The enhanced second fundamental form of $u$ is defined by
	\begin{align*}
	\mathcal A(u)&=\nabla\left(\frac{\nabla u}{|\nabla u|}\right)
	\end{align*}
and
	\begin{align*}
	|\mathcal A(u)|&=\left|\nabla\left(\frac{\nabla u}{|\nabla u|}\right)\right|\\
	&=\frac{\sqrt{|\nabla^2u|^2-|\nabla|\nabla u||^2}}{|\nabla u|}.
	\end{align*}

The enhanced second fundamental form bounds the second fundamental form of level sets, and it is not hard to see $|\mathcal A(u)|=0$ implies $\frac{\nabla u}{|\nabla u|}$ is a parallel vector field in which case the function $u$ has flat level sets.

If the second fundamental form is bounded as in the condition of Theorem \ref{ImprovementRegularity}, then the level sets can be written locally as $C^{1,\theta}$ graphs.
\subsection{Fermi Coordinates with respect to nodal sets} \label{Subsection_Fermi}
Additionally, $\delta V_t^\e$ is the first variation of the corresponding varifold; see \cite{Ilmanen1993} for details on how to consider $u_\e(x,t)$ as a general moving varifold whose density is the energy density $d\mu_{\e,t}$.

Based on the assumptions in Theorem \ref{ImprovementRegularity} of uniform bounds on the enhanced second fundamental form
	\begin{align*}
	|\mathcal A_\e|=\left|\nabla\left(\frac{\nabla u_\e}{|\nabla u_\e|}\right)\right|\leq C_0,
	\end{align*}
the enhanced second fundamental form for the rescaled solution $u^\e$ then satisfies
	\begin{align}\label{RescaledCurvatureSmall}
	|\mathcal A^\e|=\left|\nabla\left(\frac{\nabla u^\e}{|\nabla u^\e|}\right)\right|\leq C_0\e.
	\end{align}
The curvature bound $|\mathcal A^\e|\leq C_0\e$ shows that the nodal sets of the rescaled solutions $u^\e$ have second fundamental form bounded by $C_0\e$. If the nodal sets all pass through $(0,0)$ (this can be assumed in the proof, see Remark \ref{VanishingAtZero}), then up to choosing a subsequence, these nodal sets converge in $C^{1,\alpha}$ as $\e\rightarrow0$ to a flat hyperplane for every $t$. Without loss of generality we may assume this hyperplane to be $\{x_{n+1}=0\}$.

We denote by $\Gamma_{\e,t}$ and $\Gamma^\e_t$ the nodal sets of $u_\e$ and $u^\e$ at time $t$ respectively. Also we let $d_{\e,t}$ and $d^\e_t$ be signed distance to $\Gamma_{\e,t}$ and $\Gamma^\e_t$ respectively, which is positive in the direction of $\nabla u_\e$ when $|\nabla u_\e|\neq0$. Let $\mathbb R^n\supset\Omega\ni y=(y_1,...,y_n)\mapsto\Gamma^\e_0$ be a local parametrization of $\Gamma^\e_0$. By the $C^2$ convergence of $\Gamma^\e_t$ to the flat plane, we can parametrize $\Gamma^\e_t, t\in(-4,4)$ by the parametrization of $\Gamma^\e_0$ for sufficiently small $\e$ (using the nearest point projection of $\Gamma^\e_t$ to $\Gamma^\e_0$).

In a neighbourhood of $\Gamma^\e_t$ where the nearest point projection is well defined, the Fermi coordinates with respect to $\Gamma^\e_t$ are defined by $(y,z,t)\mapsto (y+z\frac{\nabla u^\e}{|\nabla u^\e|},t)\in\mathbb R^{n+1}\times\mathbb R$, where $z=\mathrm{dist}_{\Gamma^\e_t}(x)$ and $y\in\Gamma^\e_t$ is the nearest point projection of $x$ to $\Gamma^\e_t$. Moreover we denote $\Gamma^\e_{z,t}=\{d^\e_t=z\}$ for small $z$ so that $d^\e_t$ is well defined.

In Fermi coordinates, the Laplacian operator has the form
	\begin{align}\label{FermiLaplacian}
	\Delta_{\mathbb R^{n+1}}=\Delta_{\Gamma^\e_{z,t}}+\partial^2_z+H_{\Gamma^\e_{z,t}}\partial_z.
	\end{align}
where $H_{\Gamma^\e_{z,t}}=\mathrm{div}_{\Gamma^\e_{z,t}}(\partial_z)$ is the scalar mean curvature of $\Gamma^\e_{z,t}$ with respect to the normal $\partial_z$.

Here we adopt the sign convention the mean curvature vector of the sets $\Gamma^\e_{z,t}$ to be $\vec{H}_{\Gamma^\e_t}=-H_{\Gamma^\e_{z,t}}\partial_z=-\mathrm{div}_{\Gamma^\e_{z,t}}(\partial_z)\partial_z$.

\subsection{Parabolic H\"older norms}
To establish clear and concise notation for domains, we introduce the following notation. For $m\in\mathbb N^+$:
	\begin{align}\label{Notations}
	B^m_r(x_0)&=:\{y\in\mathbb R^m: |x-x_0|<r\}, x_0\in\mathbb R^m,\\\nonumber
B^m_r&=:\{y\in\mathbb R^m: |x|<r\},\\\nonumber
I&=:(-1,1)\subset\mathbb R,\\\nonumber
I_r&=:(-r,r)\subset\mathbb R, r>0.
	\end{align}
If there is no confusion from the context, we will omit the superscript $m,\e$.

The parabolic distance for $2$ points $X_1=(x_1,t_1),X_2=(x_2,t_2)\in\mathbb R^m\times\mathbb R$ is defined by $\mathrm{dist}_p(X_1,X_2)=\max(|x_1-x_2|,\sqrt{|t_1-t_2|})$

For a function $u:\mathbb R^m\times\mathbb R \rightarrow\mathbb R$ and an open set $W\subset\mathbb R^m\times\mathbb R$, we will use the parabolic H\"older norm defined by
	\begin{align}\label{HolderNorm}
	\begin{split}
&[u]_{\theta;W}=\sup_{X_1\neq X_2, X_1,X_2\in W}\frac{|u(X_1)-u(X_2)|}{\left(\mathrm{dist}_p(X_1,X_2)\right)^\theta},\\
	&\|u\|_{C^{0,\theta}(W)}=\sup_{X\in W} |u(X)|+[u]_{\theta;W},\\
	&\|u(x,t)\|_{C^{k,\theta}(W)}=\sum_{i+2j\leq k}\|\partial_x^i\partial_t^ju\|_{C^{0,\theta}(W)}.
\end{split}
	\end{align}
In particular
	\begin{align*}
	&\|u(x,t)\|_{C^{2,\theta}(W)}=\sum_{i=0}^2\sup_W|\partial_x^iu|+\sup_W|\partial_t u|+\|\partial_x^2u\|_{C^{0,\theta}(W)}+\|\partial_t u\|_{C^{0,\theta}(W)}.
	\end{align*}
We will now recall the standard parabolic Schauder estimates, which we will use later in our analysis.
\begin{lem}[{\cite[Theorem 8.2]{white2005local}}]\label{ParabolicSchauder}
Let $L=\frac{\partial}{\partial t}+\Delta_{\mathbb R^m}$ be the heat operator in $\mathbb R^m\times\mathbb R$ and let $u\in C^{2,\theta}\left(B_{r'}\times I_{(r')^2}\right)$. Then, for any $r<r^\prime$ the following estimate holds: 
	\begin{align*}
	\|u\|_{C^{2,\theta}\left(B_r\times I_{r^2}\right)}\leq C\left[\|u\|_{C^{0}\left(B_{r'}\times I_{(r')^2}\right)}+\|Lu\|_{C^{0,\theta}\left(B_{r'}\times I_{(r')^2}\right)}\right],
	\end{align*}
where $C$ depends on $r,r',m,\theta$.

\end{lem}

\section{Rigidity of the planar solution in dimension 2}\label{RigiditySection}

The $1$-d heteroclinic solution given by equation \eqref{1d} plays a role analogous to that of the static planar solution in mean curvature flow. In particular, the rigidity of such $1$-d heteroclinic solutions is an essential ingredient in the proof of curvature estimates using blow-up arguments.

Before proceeding, we recall the following rigidity theorem \cite{Mantoulidis2021} (cf. \cite{wang2014new}) for the $1$-d heteroclinic solution of the elliptic Allen--Cahn equation in any dimension. (Here, we have restated the condition from \cite{Mantoulidis2021} that the asymptotic density equals $1$ in terms of the equivalent condition that the blown-down limit is a multiplicity $1$ plane.) This will be used to give a classification of possible rescaling limits of  Allen--Cahn solutions in dimension 2 in the parabolic setting under the assumptions of Theorem \ref{CurvatureEstimates}.

\begin{thm}[Theorem 3.6 of \cite{Mantoulidis2021}]\label{EllipticRigidity}
Let $u:\mathbb R^{n+1}\rightarrow\mathbb R$ be a solution to the elliptic Allen--Cahn equation \eqref{EllipticAC}, and let its corresponding blow-down sequence be $u_\e(x)=u(\frac{x}{\e})$ for $\e\rightarrow0$. Suppose that the energy measures $d\mu_\e$ of the blow-down sequence satisfy
	\begin{align*}
	\frac{1}{\alpha}d\mu_\e=\frac{1}{\alpha}\left[\frac{\e|\nabla u_\e|^2}{2}+\frac{W(u_\e)}{\e}\right]dx\rightarrow\mathcal H^n\lfloor P
	\end{align*}
as measures, where $P\subset\mathbb R^{n+1}$ is a hyperplane.

Then, up to a rigid motion in $\mathbb R^{n+1}$, $u(x)=g(x_{n+1})$.
\end{thm}
As a consequence, we have
\begin{thm}\label{ParabolicRigidity2D}
Let $u_\e(x,t)$ be a sequence of solutions to \eqref{EAC} with $\e\rightarrow0$ in $\mathbb R^2\times\mathbb R$ satisfying 
\begin{align}\label{WillmoreBound}
\int_{\mathbb R^2}\e\left(\Delta u_\e(x,0)-\frac{W'(u_\e(x,0))}{\e^2}\right)^2dx\leq W_0
\end{align}
for some uniform $W_0>0$ as assumed in Corollary \ref{CurvatureEstimates}. Suppose that the energy measures $d\mu_\e$ satisfy
\begin{align*}
\frac{1}{\alpha}d\mu_{\e,t}\rightarrow d\mu_t
\end{align*} 
where the limit Brakke flow $d\mu_{t}=\mathcal H^n\lfloor\Sigma_t$ is non-empty and smooth for all $t\in(-r^2,r^2)$.

Consider a blow-up sequence $u^\e(x,t)=:u_\e(\e x,\e^2t)$, which satisfies the equation \eqref{AC}. Then after passing to a subsequence, the blow-up sequence $u^\e\rightarrow u^\infty$ on compact subsets, where the limit is a static 1-d heteroclinic solution $u^\infty(x,t)=g(x_2),\forall t\in\mathbb R$ up to a rotation and translation in space-time.
\end{thm}
\begin{proof}
The proof uses the sub-criticality of the Willmore type term $\int_{\mathbb R^2}\e\left(\Delta u_\e-\frac{W'(u_\e)}{\e^2}\right)^2dx$ in dimension 2 under rescalings. We have for any $R>0$ that
\begin{align*}
\int_{B_\frac{R}{\e}\cap\mathbb R^2}\left(\Delta u^\e(x,0)-W'(u^\e(x,0))\right)^2dx &=\e\cdot \int_{B_R\cap\mathbb R^2}\e\left(\Delta u^\e(x,0)-\frac{W'(u^\e(x,0))}{\e^2}\right)^2dx\\
&\leq\e W_0\rightarrow0.
\end{align*}
Thus the limit $u^\infty$ must satisfy $u_t=\Delta u-W'(u)=0$ at $t=0$, i.e. the time $t=0$ slice $u^\infty(\cdot, 0)$ satisfies the elliptic Allen--Cahn equation.

Moreover, we have
\begin{claim}\label{BlowDownToPlane2D}
Consider a blow-down sequence $u^\infty_{\e}(x,t)=u^\infty(\e x,\e^2t)$ of $u^\infty$ which satisfies the equation \eqref{EAC} with parameter $\e$, then after passing to a subsequence, the energy measures of the blow-down sequence $\frac{1}{\alpha}d\mu^\infty_{\e,t}\rightarrow\mathcal H^2\lfloor P$ on compact subsets, where the limit Brakke flow is a static flow supported on a hyperplane for all $t\in\mathbb R$.
\end{claim}
\begin{proof}[Proof of Claim \ref{BlowDownToPlane2D}]
By the sub-sequential smooth convergence of $u^\e$ to $u^\infty$ on compact subsets, for each $j>0$, we can choose $\e_j$ such that the followings are satisfied 
\begin{align}\label{PropertiesRescaledEternalFlow}
&\e_{j}<\frac{1}{j^2},\\\nonumber
&\text{$u^{\e_{j}}$ is $\frac{1}{j^2}$ close to $u^\infty$ in $C^j\left(B_{j^2}(0)\times[-j^4,j^4]\right).$}
\end{align}
Denote by $\hat u_j(x,t)=:u_{\e_{j}}(\e_{j}\cdot j\cdot x,\e_{j}^2\cdot j^2\cdot t)=u^{\e_j}(jx, j^2 t)$, which satisfies the equation \eqref{EAC} with $\e=\frac{1}{j}$. \eqref{PropertiesRescaledEternalFlow} then gives
\begin{align}\label{BlowDownCloseToPlane1}
&\text{$\tilde u_j$ is $\frac{1}{j}$ close to $u^\infty_{\frac{1}{j}}$ in $C^j\left(B_{j}(0)\times[-j^2,j^2]\right)$}.
\end{align}
by rescaling. 

Since $\frac{1}{\alpha}d\mu_{\e_j,t}\rightarrow\mathcal H^2\lfloor\Sigma_t$ for $t\in(-r^2,r^2)$ by assumption and $\frac{1}{j\e_j}\rightarrow\infty$, we then have by rescaling that the limit of the energy measures of $\tilde u_j$
\begin{align}\label{BlowDownCloseToPlane2}
\frac{1}{\alpha}\lim_{j\rightarrow\infty}d\tilde\mu_{j,t}=\lim_{j\rightarrow\infty}\mathcal H^2\lfloor\left(\frac{1}{j\e_j}\cdot\Sigma_t\right)\rightarrow\mathcal H^2\lfloor P, \forall t\in\mathbb R
\end{align}
as measures.

Because $j$ is arbitrary, combining \eqref{BlowDownCloseToPlane1}, \eqref{BlowDownCloseToPlane2} and the triangle inequality, we have that the energy measures 
\begin{align*}
\frac{1}{\alpha}du^\infty_{\frac{1}{j},t}\rightarrow\mathcal H^2\lfloor P, \forall t\in\mathbb R,
\end{align*} 
namely $u^\infty$  blows down to a flat plane after passing to a subsequence.
\end{proof}
Now with this claim applied to $t=0$ slice, we know that $u^\infty(\cdot, 0)$ satisfies the elliptic Allen Cahn equation and blows down to a plane. By the rigidity in the elliptic case (Theorem \ref{EllipticRigidity}), it is the $1$-d heteroclinic solution with a flat slice for that particular time, and thus the whole eternal solution is a static $1$-d heteroclinic solution by the uniqueness of the Cauchy problem.
\end{proof}

\begin{rem}
The rigidity theorem discussed here has been generalised to all dimensions $n$ and without the assumptions of the bound \eqref{WillmoreBound} in a recent work \cite[Theorem 1.3]{Nguyen2020} if an appropriate scale invariant renormalised energy is sufficiently close to $1$.
\end{rem}

\section{The Approximate Solution}\label{ApproximateSolution}

To construct an approximate solution, we use the zero sets of a parabolic Allen--Cahn equation given by equation \eqref{AC}. Specifically, we compose the local distance function to its nodal sets with the $1$-d heteroclinic solution. Our goal is to demonstrate that this approximation is well-controlled provided that the nodal sets are singly sheeted.

Similar to section 9 of \cite{Wang2019a}, we choose $\bar g$ to be a smooth cutoff approximation at infinity of the $1$-d heteroclinic solution $g$ with well controlled errors
 	\begin{align}\label{DefBarG}
	\bar g(x)=\zeta\left(\frac{x}{3|\log \e|}\right)g(x)+\left(1-\zeta\left(\frac{x}{3|\log \e|}\right)\right)\mathrm{sgn}(x)
	\end{align}
 where $\zeta$ is a smooth cutoff function supported in $(-2,2)$ with $\zeta\equiv1$ in $(-1,1)$ and $|\zeta'|+|\zeta''|\leq16$, and $\mathrm{sgn}=\frac{x}{|x|}$ is the sign function.
We have
	\begin{align*}
	\bar g''=W'(\bar g)+\bar\eta
	\end{align*}
with
	\begin{align}
	\begin{split}\label{CutoffError}
&\spt(\bar\eta)\subset\{3|\log\e|\leq|x|\leq6|\log\e|\}\\
	&|\bar\eta|+|\bar\eta'|+|\bar\eta''|\leq O(\e^3)\\
	&\int\bar g'^2=\alpha+O(\e^3)\\
	&\sup_{\mathbb R}|g-\bar g|=O(\e).
\end{split}
	\end{align}
For each $h\in C^2\left((\Gamma^\e_{0,0}\cap B_2)\times(-2,2)\right)$, we define
	\begin{align}\label{eqn_best_approx}
	g^{\e,*}(y,z,t)=\bar g(z-h(y,t))=\bar g(d^\e_t-h(y,t)).
	\end{align}
Here the function $h$ is used to obtain an optimal approximation to offset the effect from mean curvature of the nodal sets $\Gamma^\e_{0,t}$. We have
\begin{prop}[c.f. \cite{Wang2019a} Proposition 9.1]
If the nodal sets $\Gamma^\e_{0,t}$ have uniformly bounded second fundamental form (as assumed in Theorem \ref{ImprovementRegularity}), then here exists an $h$ with $|h|\ll1$ such that
\begin{align}\label{Ortho}
	\int (u^\e-g^{\e,*})\bar g'(z-h(y,t)) dz=0,
	\end{align}
	for each $t$.
\end{prop}
\begin{rem}
The proof is essentially the same as in \cite[Proposition 9.1]{Wang2019a}. The condition that the level sets have uniformly bounded second fundamental form guarantees that the rescaled solutions are arbitrarily close to a $1$-d heteroclinic solution $g$ with flat level sets. (See also Theorem \ref{GradientUpperLowerBound} for the argument of showing that the blow-up limit being the $1$-d solution under the same assumptions.)
\end{rem}
We denote by $\phi^\e=u^\e-g^{\e,*}$. And for simplicity, we also denote
	\begin{align}\label{SimpleDerivative}
	(g^{\e,*})'=g'(z-h(y,t)), (g^{\e,*})''=g''(z-h(y,t)).
	\end{align}

We compute in the $(y,z,t)$ coordinates
 	\begin{align}\label{EquationInFermiCoordinate}
	&\frac{\partial g^{\e,*}}{\partial t}-\Delta g^{\e,*}\\\nonumber
&= (g^{\e,*})'\cdot \left[-\left\langle \frac{\partial X_{\Gamma^\e_{z,t}}}{\partial t},\nu_{\Gamma^\e_{0,t}}\right\rangle-\frac{\partial h}{\partial t}\right]- (g^{\e,*})''-H_{\Gamma^\e_{z,t}} (g^{\e,*})'+ (g^{\e,*})'\Delta_{\Gamma^\e_{z,t}} h\\\nonumber
&- (g^{\e,*})''|\nabla h|^2.
	\end{align}

We will see later the normal velocity term $\left\langle \frac{\partial X_{\Gamma^\e_{z,t}}}{\partial t},\nu_{\Gamma^\e_{0,t}}\right\rangle$ cancels out the mean curvature term of the nodal sets up to small error as $\e\rightarrow0$ by the convergence to the mean curvature flow for the unscaled equation as $\e\rightarrow0$.

We compute the equation for $\phi^\e$ as follows
	\begin{align}\label{Difference}
	\begin{split}
&\left(\frac{\partial}{\partial t}-\Delta\right)\phi^\e\\
	&= \left(\frac{\partial}{\partial t}-\Delta_{\Gamma^\e_{z,t}}-\partial^2_z-H_{\Gamma^\e_{z,t}}\partial_z\right)\phi^\e\\
	&= -W'(\phi^\e+g^{\e,*})+W'(g^{\e,*})+\bar\eta-\bar g'(z-h)\cdot\left[-\left\langle \frac{\partial X_{\Gamma^\e_t}}{\partial t},\nu_{\Gamma^\e_t}\right\rangle-\frac{\partial h}{\partial t}\right]\\
	&+H_{\Gamma^\e_{z,t}}\bar g'(z)-\bar g'(z-h)\Delta_{\Gamma^\e_{z,t}} h+\bar g''(z-h)|\nabla h|^2\\
	&= -[W'(\phi+g^{\e,*})-W'(g^{\e,*})]+\left[(g^{\e,*})'\left( \frac{\partial h}{\partial t}-\Delta_{\Gamma^\e_{z,t}}h+H_{\Gamma^\e_{z,t}}\right)\right]+[(g^{\e,*})''|\nabla h|^2]\\
	&+\left[(g^{\e,*})'\left\langle \frac{\partial X_{\Gamma^\e_t}}{\partial t},\nu_{\Gamma^\e_t}\right\rangle\right]+\bar\eta\\
	&= I+II+III+IV+\bar\eta.
\end{split}
	\end{align}

The Term I is given by $-W'(\phi^\e+g^{\e,})+W'(g^{\e,})=-W''(g^{\e,*})\phi^\e+\mathcal R(\phi^\e)$, and its H\"older norm is bounded by the H\"older norm of $\phi_\e$. The remainder term $\mathcal R(\phi^\e)$ is a polynomial in $\phi^\e$ whose linear and constant terms vanish.

The Term III is bounded by the $C^{2,\theta}$ norm of $h$, which, in turn, is bounded by the $C^{2,\theta}$ norms of $\phi^\e$ via an interpolation inequality (cf. \cite[Lemma 9.6]{Wang2019a}).

We will estimate the H\"older norm of Term II + Term IV and demosntrate they are sufficiently small in the following section.


We will conclude this section with the following two lemmas, whose proofs will be presented in Appendix Section \ref{Appendix}. The first lemma provides an estimate of the H\"older norm of the difference between $u_\e$ and the flat $1$-d solution, and the second lemma gives error estimates for the geometries of $\Gamma^\e_{z,t}$ and $\Gamma^\e_{0,t}=\Gamma^\e_t$.
\begin{lem}\label{LemmaEstimateDifferenceToFlatSolution}
Suppose $u^\e$ satisfies $u^\e(0,0)=0$, then for any $0<r<r^\prime$, then up to a rotation about the space origin, the H\"older norm of the difference between $u^\e$ and the $1$-d heteroclinic solution is bound as follows:
	\begin{align}\label{HolderSmallnessRescaled}
	&\|u^\e-\tilde g\|_{C^{k,\theta}(B_r(x_0,0)\times I_{r^2})}\leq O(\e)+O(\|\phi^\e\|_{C^{2,\theta}(B_{r'}\times I_{(r')^2})}),\\\nonumber
&\forall k\in\mathbb N, \theta\in(0,1),
	\end{align}
where $\tilde g(x)=\tilde g(x_1,...,x_{n+1})=g(x_{n+1})$ and $C^{k,\theta}$ is the parabolic H\"older norm (see \eqref{HolderNorm} for the definition).

In particular, since $\frac{\partial}{\partial t}\tilde g=0$, the H\"older norms which contain at least one order of $t$ derivative satisfy
	\begin{align}
	\left \|\frac{\partial}{\partial t}u^\e\right \|_{C^{k,\theta}(B_r(x_0,0)\times I_{r^2})}\leq O(\e)+O(\|\phi^\e\|_{C^{2,\theta}(B_{r'}\times I_{(r')^2})}).
	\end{align}
\end{lem}



Since the second fundamental form satisfies $A_{\Gamma^\e_{z,t}(y)}=(I-zA_{\Gamma^\e_{0,t}(y)})^{-1}A_{\Gamma^\e_{0,t}(y)}$, we also have the error estimates in $z$ for the second fundamental forms and Laplacian.
\begin{lem}\label{LemmaErrorInZ}
Suppose $u^\e$ satisfies $u^\e(0,0)=0$, then for any $0<r<r^\prime$ the error of second fundamental form and its time derivatives in $z$ direction is bounded by
	\begin{align}\label{CurvatureError}
	&\sup_{y\in B_r^n}|A_{\Gamma^\e_{z,t}(y)}-A_{\Gamma^\e_{0,t}(y)}|\leq\sup_{y\in B_r^n}|z||A_{\Gamma^\e_{0,t}(y)}|^2=O(\e^2)\\\label{CurvatureDerivativeError}
&\sup_{y\in B_r^n}\left|\frac{\partial}{\partial t}\left(A_{\Gamma^\e_{z,t}(y)}-A_{\Gamma^\e_{0,t}(y)}\right)\right|=O(\e^2)+O(\|\phi^\e\|^2_{C^{2,\theta}(B_{r'}\times I_{(r')^2})}).
	\end{align}
In particular, since the H\"older norm is bounded by the Lipschitz norm in time up to a uniform constant, we have
	\begin{align*}
	\sup_{y\in B_r^n}\|A_{\Gamma^\e_{z,\cdot}(y)}-A_{\Gamma^\e_{0,\cdot}(y)}\|_{C^\frac{\theta}{2}(I_{r^2})}\leq O(\e^2)+O\left(||\phi^\e||^2_{C^{2,\theta}(B_{r'})\times I_{(r')^2}}\right).
	\end{align*}
Similarly, for any $\varphi\in C^2\left(B_{r'}^n\times I_{(r')^2}\right)$, we can compute the error in the $z$ direction of $\Delta_{\Gamma^\e_{z,t}}\varphi$ and its H\"older norm in time and obtain
	\begin{align}\label{LaplacianError}
	\sup_{y\in B_r^n}|\Delta_{\Gamma^\e_{z,t}}\varphi(y,t)-\Delta_{\Gamma^\e_{0,t}}\varphi(y,t)|&\leq\left(\sup_{B_r^n}\e|z| (|\nabla\varphi(\cdot,t)|+|\nabla^2\varphi(\cdot,t)|)\right)\\\nonumber
&\leq O(\e^2)+O\left(||\varphi||^2_{C^{2,\theta}\left(B_r^n\times I_{r^2}\right)}\right)\\\label{LaplacianHolderError}
\sup_{y\in B_r^n}\|\Delta_{\Gamma^\e_{z,\cdot}}\varphi(y,\cdot)-\Delta_{\Gamma^\e_{0,\cdot}}\varphi(y,\cdot)\|_{C^\frac{\theta}{2}(I_{r^2})}&\leq O(\e^2)+O\left(\|\phi^\e\|^2_{C^{2,\theta}(B_{r'}\times I_{(r')^2})}\right)\\\nonumber
&+O\left(||\varphi||^2_{C^{2,\theta}\left(B_{r'}^n\times I_{(r')^2}\right)}\right),
	\end{align}
for $\theta\in(0,1)$.
\end{lem}
\begin{rem}\label{VanishingAtZero}
Notice that we do not assume $u^\e(0,0)=0$ in Theorem \ref{ImprovementRegularity}. However, in the proof of the uniform curvature bound, we can perform a parabolic translation in space-time so that this condition holds for the whole sequence without changing the $C^{2,\alpha}$ bound.
\end{rem}
We also record that the difference of metrics in space can also be bounded (c.f. \cite[Section 8]{Wang2019a})
\begin{align}\label{MetricError}
&\left(g_{ij}\right)_{\Gamma^\e_{z,t}}=\left(g_{ij}\right)_{\Gamma^\e_{0,t}}+O\left(|A|_{\Gamma^\e_{0,t}}\right)=\left(g_{ij}\right)_{\Gamma^\e_{0,t}}+O(\e)\\\nonumber
&|\nabla_yg_{ij}(y,z,t)|+|\nabla_yg^{ij}(y,z,t)|=O(\e).
\end{align}

\section{Derivation of the equation and estimate for term II+IV}\label{DerivationEquation}\label{DerivationEquation}

In this section, we derive estimates that are analogous to those in Section 10 of Wang and Wei's work \cite{Wang2019a}. In that paper, the authors developed a Toda system to model interactions between different layers. However, in our case, the entropy assumption allows us to assume the nodal sets converge to a single layer, which is guaranteed by our entropy bound condition. This assumption leads to significant simplifications in the equations.
From now on we will drop the superscripts $\e$ in the equations if there is no confusion. Throughout this section, we will use the following notation
	\begin{align*}
	u(x,t)=:u^\e(x,t)=u_\e(\e x,\e^2t)
	\end{align*}
which are solutions of \eqref{AC} obtained by rescaling solutions $u_\e$ of \eqref{EAC}. Similarly we denote by
	\begin{align*}
	\Gamma_{z,t}&=:\Gamma^\e_{z,t},\\
	\Gamma_t&=:\Gamma^\e_t=\Gamma^\e_{0,t},\\
	 f_t&=:f^\e_t=f^\e(\cdot,t),\\
	 g^*&=:g^{\e,*},\\
	\phi&=:\phi^\e=u^\e-g^{\e,*}.
	\end{align*}
As a consequence of \eqref{CutoffError}, we get 
\begin{align}\label{g*Error}
&W'(\phi+g^*)-W'(g^*)= W''(g^*)(g^*)'\phi+\mathcal R(\phi)\\\nonumber
&\int (g^*)'\partial_{zz}\phi=\int (g^*)'''\phi=\int W''(g^*)(g^*)'\phi+O(\e^2).
\end{align}
By multiplying \eqref{Difference} by $(g^*)'(y,z,t)=\bar g'(z-h(y,t))$ and integrating in the spatial direction normal to the nodal sets (notice $\bar g$ is compactly supported, so the integral is well defined), we get
	\begin{align}\label{IntegralEquation}
	\begin{split}
&\int (g^*)'\left(\frac{\partial}{\partial t}-\Delta\right)\phi dz\\
	&=\int (g^*)'\left(\frac{\partial}{\partial t}-\Delta\right)(u-g^*) dz\\
	&= \int (g^*)'\left(\frac{\partial}{\partial t}-\Delta_{\Gamma_{z,t}}-\partial^2_z-H_{\Gamma_{z,t}}\partial_z\right)(u-g^*) dz\\
	&= -\int (g^*)'[W'(\phi+g^*)-W'(g^*)] dz+\int (g^*)'\left[(g^*)''|\nabla_{\e,z,t} h|^2\right] dz\\
	&+\int [(g^*)']^2\left[\frac{\partial h}{\partial t}-\Delta_{\Gamma_{z,t}}h+H_{\Gamma_{z,t}}+\left\langle \frac{\partial X_{\Gamma_t}}{\partial t},\nu_{\Gamma_t}\right\rangle\right] dz+\int (g^*)'\bar\eta dz,
\end{split}
	\end{align}
where we used \eqref{EquationInFermiCoordinate} to compute $\left(\frac{\partial}{\partial t}-\Delta\right)g^*$.

To obtain improved estimates, we make use of the orthogonality condition \eqref{Ortho} to offset the error in vertical direction of the approximate solution $\phi$ (using notation \eqref{SimpleDerivative})
	\begin{align}\label{Orthogonality}
	\begin{split}
&\int [u(y,z,t)-g^*(y,z,t)]\bar g'(z-h(y,t)) dz=\int \phi \bar g'(z-h)
 dz\\
	&=\int \phi (g^*)'dz=0.
\end{split}
	\end{align}

Differentiating once \eqref{Orthogonality} in tangential direction (the coordinate $y$ with respect to the Fermi coordinate) we get
\begin{equation*}
\int \phi_{y_i}(g^*)'dz-h_{y_i}\int \phi (g^*)'' dz=0.
\end{equation*}
Differentiating again we have
	\begin{align*}
	&\int \left(\frac{\partial^2\phi}{\partial y_i\partial y_j}(g^*)'-\frac{\partial \phi}{\partial y_i}(g^*)''\frac{\partial h}{\partial y_j}-\frac{\partial \phi}{\partial y_j}(g^*)''\frac{\partial h}{\partial y_i}-\phi (g^*)''\frac{\partial^2 h}{\partial y_i\partial y_j}+\phi (g^*)'''\frac{\partial h}{\partial y_i}\frac{\partial h}{\partial y_j}\right) dz\\&=0.
	\end{align*}
And thus by taking the trace, we get
	\begin{align}\label{LeftLaplacian}
	\begin{split}
&\int \Delta_{\Gamma_t}\phi (g^*)' dz\\
	&= \Delta_{\Gamma_t}h\int \phi (g^*)'' dz-2\int \left\langle\nabla_{\Gamma_t}\phi,\nabla_{\Gamma_t}h\right\rangle(g^*)'' dz\\
	&-|\nabla_{\Gamma_t} h|^2\int \phi (g^*)''' dz,
\end{split}
	\end{align}
where $g_{\Gamma_t}$ is the induced Riemannian metric on $\Gamma_t$

Moreover, differentiating the orthogonality condition \eqref{Ortho} with respect to time $t$ and integrating by parts, we get

	\begin{align}\label{LeftTime}
	\begin{split}
\int \frac{\partial\phi}{\partial t}(g^*)' dz&= -\int \phi (g^*)''\left(-\left\langle \frac{\partial X_{\Gamma_t}}{\partial t},\nu_{\Gamma_t}\right\rangle-\frac{\partial h}{\partial t}\right) dz\\
	&= \left(\left\langle \frac{\partial X_{\Gamma_t}}{\partial t},\nu_{\Gamma_t}\right\rangle+\frac{\partial h}{\partial t}\right)\int \phi (g^*)'' dz.
\end{split}
	\end{align}

Finally, we also have that the H\"older norms of $\phi$ bound H\"older norms of $h$ (c.f. \cite[9.6]{Wang2019a})
\begin{lem}
	\begin{align}\label{PhiBoundsh}
	\|h(\cdot, t)\|_{C^{k,\theta}(B^n_r(y_0))}&\leq O(||\phi(\cdot,t)||_{C^{k,\theta}(B_r(y_0,0))}),\\\nonumber
\|h\|_{C^{k,\theta}(B^n_r(y_0)\times I_{r^2})}&\leq O(||\phi||_{C^{k,\theta}(B_r(y_0,0)\times I_{r^2})}),
	\end{align}
for any $(y_0,0,t)\in\Gamma_t$ in the Fermi coordinate and $k=0,1,2$.
\end{lem}
\begin{proof}
Since
	\begin{align*}
	\phi(y,z,t)=u(y,z,t)-g^*(y,z,t)=u(y,z,t)-\bar g\left(z-h(y,t)\right)
	\end{align*}
and $u(y,0,t)=u^\e(y,0,t)=0$ in the Fermi coordinate, we have
	\begin{align*}
	\phi(y,0,t)=-\bar g\left(-h(y,t)\right).
	\end{align*}
Notice that $|h|\ll1$ and so there is a uniform $C>0$ such that
	\begin{align}\label{GradientBound_g}
	\frac{1}{C}<\bar{g}^{\prime}(-h(y, t))\leq C,
	\end{align}
so we have
	\begin{align*}
	|h(y,t)|=\bar g^{-1}\left(-\phi(y,t)\right)\leq C|\phi(y,t)|.
	\end{align*}
Differentiating in the spatial $y$ variable and time $t$ variable respectively gives
	\begin{align*}
	\nabla_0\phi(y,0,t)&=\bar g'\left(-h(y,t)\right)\nabla h(y,t)\\
	\nabla^2_0\phi(y,0,t)&=-\bar g''\left(h(y,t)\right)\nabla h(y,t)\otimes\nabla h(y,t)+\bar g'\left(h(y,t)\right)\nabla^2h\\
	\frac{\partial}{\partial t}\phi(y,0,t)&=\bar g'\left(-h(y,t)\right)\frac{\partial}{\partial t}h(y,t).
	\end{align*}
Combining with \eqref{GradientBound_g}, we then have the desired bounds.
\end{proof}
The term from II+IV comes out of the integral by the error estimates \eqref{CurvatureError}, \eqref{LaplacianError} and \eqref{PhiBoundsh}. In $B_r(y_0)\times I$,
	\begin{align*}
	&\int [(g^*)']^2\left[\frac{\partial h}{\partial t}-\Delta_{\Gamma_{z,t}}h+H_{\Gamma_{z,t}}+\left\langle \frac{\partial X_{\Gamma_t}}{\partial t},\nu_{\Gamma_t}\right\rangle\right] dz\\
	&= \alpha\left[\frac{\partial h}{\partial t}-\Delta_{\Gamma_{0,t}}h+H_{\Gamma_{0,t}}+\left\langle \frac{\partial X_{\Gamma_t}}{\partial t},\nu_{\Gamma_t}\right\rangle\right]+O(\e^2)+O(\e||h||_{C^{2,\theta}(B^n_r(y_0)\times I_{r^2})})\\
	&= \alpha\left[\frac{\partial h}{\partial t}-\Delta_{\Gamma_t}h+H_{\Gamma_t}+\left\langle \frac{\partial X_{\Gamma_t}}{\partial t},\nu_{\Gamma_t}\right\rangle\right]+O(\e^2)+O(||\phi||^2_{C^{2,\theta}(B_r(y_0,0)\times I_{r^2})}),
	\end{align*}
where $\alpha$ is the total energy for the $1$-d heteroclinic solution. The additional term not in Wang--Wei is
	\begin{align*}
	&\int [(g^*)']^2\left(\frac{\partial h}{\partial t}+\left\langle \frac{\partial X_{\Gamma_t}}{\partial t},\nu_{\Gamma_t}\right\rangle\right) dz\\
	&= \left(\frac{\partial h}{\partial t}+\left\langle \frac{\partial X_{\Gamma_t}}{\partial t},\nu_{\Gamma_t}\right\rangle\right)\int_{-\infty}^\infty[(g^*)']^2 dz\\
	&= \alpha\cdot\frac{\partial h}{\partial t}+\alpha \left\langle \frac{\partial X_{\Gamma_t}}{\partial t},\nu_{\Gamma_t}\right\rangle+O(\e^3)
	\end{align*}
by the error control in the cutoff $\bar g$ of $g$ in \eqref{CutoffError}.

\subsection{Sup norm of II+IV}
Using integration by parts and \eqref{CutoffError}, equation \eqref{IntegralEquation} can be written as
	\begin{align*}
	&\int (g^*)'\left(\frac{\partial}{\partial t}-\Delta\right)\phi dz\\
	&= -\int (g^*)'[W'(\phi+g^*)-W'(g^*)]dz+\int \bar g'\bar g''|\nabla_{\Gamma_{z,t}} h|^2dz\\&+\int [(g^*)']^2\left[\frac{\partial h}{\partial t}-\Delta_{\Gamma_{z,t}}h+H_{\Gamma_{z,t}}+\left\langle \frac{\partial X_{\Gamma_t}}{\partial t},\nu_{\Gamma_t}\right\rangle\right] dz+O(\e^2).
	\end{align*}
Utilising \eqref{LeftLaplacian} and \eqref{LeftTime}, we compute in $B_r(y_0)\times I_{r^2}$, the left hand side (LHS)
	\begin{align*}
	&LHS\\
	&=\int (g^*)'\left(\frac{\partial}{\partial t}-\Delta\right)\phi dz\\
	&= \left(\left\langle \frac{\partial X_{\Gamma_t}}{\partial t},\nu_{\Gamma_t}\right\rangle+\frac{\partial h}{\partial t}\right)\int \phi (g^*)'' dz-\int (g^*)'\left(\Delta_{\Gamma_{z,t}}+\partial^2_z+H_{\Gamma_{z,t}}\partial_z\right)\phi dz\\
	&= \left(\left\langle \frac{\partial X_{\Gamma_t}}{\partial t},\nu_{\Gamma_t}\right\rangle+\frac{\partial h}{\partial t}\right)\int \phi (g^*)'' dz-\Delta_{\Gamma_t}h\int \phi (g^*)'' dz+O(||\phi||^2_{C^{2,\theta}(B_r(y_0,0)\times I_{r^2})})\\
	&-\int \phi_{zz}(g^*)' dz-\int(g^*)'\left(\Delta_{\Gamma_{z,t}}-\Delta_{\Gamma_{0,t}}\right)\phi dz-\int H_{\Gamma_{0,t}}\phi_z (g^*)' dz-\int \left(H_{\Gamma_{z,t}}-H_{\Gamma_{z,t}}\right)\phi_z (g^*)' dz\\
	&= \left(\left\langle \frac{\partial X_{\Gamma_t}}{\partial t},\nu_{\Gamma_{\e,t}}\right\rangle+\frac{\partial h}{\partial t}\right)\int \phi (g^*)'' dz-\Delta_{\Gamma_t}h\int \phi (g^*)'' dz\\
	&-\int W''(g^*)(g^*)'\phi dz+\int H_{\Gamma_{z,t}}\phi (g^*)'' dz+O(\e^2)+O(||\phi||^2_{C^{2,\theta}(B_r(y_0,0)\times I_{r^2})})\\
	&= \left(\left\langle \frac{\partial X_{\Gamma_{\e,t}}}{\partial t},\nu_{\Gamma_{\e,t}}\right\rangle+\frac{\partial h}{\partial t}\right)\int \phi (g^*)'' dz-\Delta_{\Gamma_{\e,0,t}}h\int \phi (g^*)'' dz\\
	&-\int W''(g^*)(g^*)'\phi dz+H_{\Gamma_{\e,0,t}}\int \phi (g^*)'' dz+O(\e^2)+O(||\phi||^2_{C^{2,\theta}(B_r(y_0,0)\times I_{r^2})})\\
	&= \left(H_{\Gamma_{\e,0,t}}+\left\langle \frac{\partial X_{\Gamma_{\e,t}}}{\partial t},\nu_{\Gamma_{\e,t}}\right\rangle+\frac{\partial h}{\partial t}-\Delta_{\Gamma_{\e,0,t}}h\right)\int \phi(g^*)'' dz\\
	&-\int W''(g^*)(g^*)'\phi dz+O(\e^2)+O(||\phi||^2_{C^{2,\theta}(B_r(y_0,0)\times I_{r^2})}).
	\end{align*}
In the above we also used the expression of the Laplacian in Fermi coordinates \eqref{FermiLaplacian} in the second equality and \eqref{g*Error} in the third equality. The change of sign of the mean curvature term is due to integration by parts and we are able to take the mean curvature term out of integral is due to the error estimates \eqref{CurvatureError}.

Next we compute in $B^n_r(y_0)\times I$, the right hand side (RHS)
	\begin{align*}
	&RHS\\
	&= -\int W''(g^*)(g^*)'\phi dz+O(||\phi||^2_{C^{2,\theta}(B_r(y_0,0)\times I)})+\int (g^*)'(g^*)''|\nabla_{\Gamma_{z,t}} h|^2 dz\\&+\int [(g^*)']^2\left[\frac{\partial h}{\partial t}-\Delta_{\Gamma_{z,t}}h+H_{\Gamma_{z,t}}+\left\langle \frac{\partial X_{\Gamma_t}}{\partial t},\nu_{\Gamma_t}\right\rangle\right] dz\\
	&= -\int W''(g^*)(g^*)'\phi dz+O(||\phi||^2_{C^{2,\theta}(B_r(y_0,0)\times I)})+\int (g^*)'(g^*)''|\nabla_{\Gamma_{z,t}} h|^2 dz\\&+\alpha\left[\frac{\partial h}{\partial t}-\Delta_{\Gamma_{0,t}}h+H_{\Gamma_{0,t}}+\left\langle \frac{\partial X_{\Gamma_t}}{\partial t},\nu_{\Gamma_t}\right\rangle\right]+\int[(g^*)']^2\left[-(\Delta_{\Gamma_{z,t}}-\Delta_{\Gamma_{0,t}})h+(H_{\Gamma_{z,t}}-H_{\Gamma_{0,t}})\right]dz\\
	&= -\int W''(g^*)(g^*)'\phi dz+O(||\phi||^2_{C^{2,\theta}(B_r(y_0,0)\times I)})+\int (g^*)'(g^*)''|\nabla_{\Gamma_{z,t}} h|^2 dz\\&+\alpha\left[\frac{\partial h}{\partial t}-\Delta_{\Gamma_{0,t}}h+H_{\Gamma_{0,t}}+\left\langle \frac{\partial X_{\Gamma_t}}{\partial t},\nu_{\Gamma_t}\right\rangle\right]+O(\e^2),
	\end{align*}
where we used the Taylor expansion of $W'$ and Lemma \ref{LemmaErrorInZ}.

Combining the above, we obtain
	\begin{align}\label{KeyTerm}
	\begin{split}
&\left[\frac{\partial h}{\partial t}-\Delta_{\Gamma_t}h+H_{\Gamma_{t}}+\left\langle \frac{\partial X_{\Gamma_t}}{\partial t},\nu_{\Gamma_t}\right\rangle\right]\left[\int \phi (g^*)'' dz-\alpha\right]+O(\e^2)\\
	&= O(\e^2)+\int (g^*)'(g^*)''|\nabla_{\Gamma_{z,t}} h|^2 dz+\|h\|^2_{C^{2,\theta}(B^n_r(y_0)\times I_{r^2})}\\
	&= O\left(\e^2\right)+O\left(\|\phi\|^2_{C^{2,\theta}(B_r(y_0,0)\times I_{r^2})}\right),
\end{split}
	\end{align}

In the last equality above, we used the fact that $\|h\|_{C^{2,\theta}}$ is controlled by $\|\phi\|_{C^{2,\theta}}$ as shown in \eqref{PhiBoundsh}, and we applied Cauchy's inequality to the middle term. The supremum norm estimates are obtained by integrating by parts and using the fact that the integral of $\bar{g}$ and its derivatives are uniformly bounded, independent of $\e$. Specifically, we have:
	\begin{align*}
	&\sup_{B_r(y_0)\times I_{r^2}}|II+IV|\\
	&= \sup_{B_r(y_0)\times I_{r^2}}\left|\left[\frac{\partial h}{\partial t}-\Delta_{\Gamma_{z,t}}h+H_{\Gamma_{z,t}}+\left\langle \frac{\partial X_{\Gamma_t}}{\partial t},\nu_{\Gamma_t}\right\rangle\right]\right|\\
	&\leq \sup_{B_r(y_0)\times I_{r^2}}\left|\left[\frac{\partial h}{\partial t}-\Delta_{\Gamma_{t}}h+H_{\Gamma_{t}}+\left\langle \frac{\partial X_{\Gamma_t}}{\partial t},\nu_{\Gamma_t}\right\rangle\right]\right|+\sup_{B_r(y_0)\times I_{r^2}}\left|\Delta_{\Gamma_{tz,}}h-\Delta_{\Gamma_{t}}h\right|\\
	&+\sup_{B_r(y_0)\times I_{r^2}}\left|\Delta_{\Gamma_{z,t}}h-\Delta_{\Gamma_{t}}h\right|\\
	&\leq O(\e^2)+O\left(\|\phi\|^2_{C^{2,\theta}(B_r(y_0,0)\times I_{r^2})}\right),
	\end{align*}
where we used Lemma \ref{LemmaErrorInZ} and \eqref{PhiBoundsh} to bound the last two terms in the third line.

\subsection{H\"older norm of II+IV}
Again by rewriting equation \eqref{IntegralEquation} using the orthogonality conditions \eqref{LeftLaplacian} and \eqref{LeftTime}, we have
	\begin{align*}
	&\int (g^*)'\left(\frac{\partial}{\partial t}-\Delta_{\Gamma_{z,t}}-\partial^2_z-H_{\Gamma_{z,t}}\partial_z\right)\phi dz\\
	&= \int [\Delta_{\Gamma_{0,t}}-\Delta_{\Gamma_{z,t}}]\phi (g^*)' dz\\
	&+\left(\left\langle \frac{\partial X_{\Gamma_t}}{\partial t},\nu_{\Gamma_t}\right\rangle+\frac{\partial h}{\partial t}\right)\int \phi (g^*)'' dz-\Delta_{\Gamma_{0,t}}h\int \phi (g^*)'' dz\\
	&+2\int \left\langle\frac{\partial\phi}{\partial y_i}, \frac{\partial h}{\partial y_j}\right\rangle_{g_{\Gamma_{y,0}}}(g^*)'' dz+|\nabla_{\Gamma_{0,t}}h|^2\int \phi (g^*)''' dz-\int (g^*)'(\partial_{zz}\phi+H_{\Gamma_{z,t}}\partial_z\phi) dz+O(\e^2)\\
	&= -\int (g^*)'[W'(\phi+g^*)-W'(g^*)] dz+\int [(g^*)']^2\left[\frac{\partial h}{\partial t}-\Delta_{\Gamma_{z,t}}h+H_{\Gamma_{z,t}}+\left\langle \frac{\partial X_{\Gamma_t}}{\partial t},\nu_{\Gamma_t}\right\rangle\right] dz\\
	&+\int (g^*)'[(g^*)''|\nabla_{z,t} h|^2] dz+O(\e^2),
	\end{align*}
where $|\int (g^*)'\bar\eta dz|=|\int (g^*)\bar\eta' dz|=O(\e^2)$ by the definition of $\bar\eta$ in \eqref{CutoffError}. Further simplification and some integration by parts gives
	\begin{align}\label{EquationForHolder}
	\begin{split}
&\int [\Delta_{\Gamma_{0,t}}-\Delta_{\Gamma_{z,t}}]\phi (g^*)' dz\\
	&+\left(\left\langle \frac{\partial X_{\Gamma_{\e,t}}}{\partial t},\nu_{\Gamma_{\e,t}}\right\rangle+\frac{\partial h}{\partial t}\right)\int \phi(g^*)'' dz-\Delta_{\Gamma_{0,t}}h\int \phi (g^*)'' dz\\
	&+2\int \left\langle \frac{\partial\phi}{\partial y_i}, \frac{\partial h}{\partial y_j}\right\rangle_{g_{\Gamma_{y,0}}}(g^*)'' dz+|\nabla_{\Gamma_{0,t}}h|^2\int \phi (g^*)''' dz-\int (g^*)'H_{\Gamma_{z,t}}\partial_z\phi dz\\
	&= O(\e^2)+\int(g^*)'\mathcal R(\phi)dz+\left[\frac{\partial h}{\partial t}-\Delta_{\Gamma_{0,t}}h+H_{\Gamma_{0,t}}+\left\langle \frac{\partial X_{\Gamma_t}}{\partial t},\nu_{\Gamma_t}\right\rangle\right]\left(\int g'^2 dz\right)\\
	&-\int [H_{\Gamma_{z,t}}-H_{\Gamma_{0,t}}] [(g^*)']^2 dz+\int [\Delta_{\Gamma_{z,t}}h-\Delta_{\Gamma_{0,t}}h] [(g^*)']^2 dz+\int (g^*)'(g^*)''|\nabla_{z,t} h|^2 dz.
\end{split}
	\end{align}
where we used \eqref{g*Error}. Here the remainder term $\mathcal R(\phi)$ is a polynomial in $\phi$ with vanishing constant and linear term.

\subsubsection{H\"older estimates in space}

We will estimate the spatial H\"older norms in \eqref{EquationForHolder} term by term and use the notation in \eqref{Notations} for domains.

Since $\mathcal R(\phi)$ has vanishing constant and linear term, we have
	\begin{align*}
	\sup_{t\in I_{r^2}}\left\|\int(g^*)'\mathcal R(\phi)dz\right\|_{C^\theta(B^n_r(y_0))}\lesssim \|\phi\|^2_{C^{2,\theta}(B_r(y_0,0)\times I_{r^2})}.
	\end{align*}
Since the $C^{k,\theta}$ norms of $\phi$ control the $C^{k,\theta}$ norms of $h$ by \eqref{PhiBoundsh}, we have
	\begin{align*}
	&\sup_{t\in I_{r^2}}\left\|\Delta_{\Gamma_{0,t}}h(\cdot,t)\left(\int\phi(\cdot,z,t) (g^*)'' (\cdot,z,t)dz\right)\right\|_{C^\theta(B_r(y_0))}\\
	&\lesssim \sup_{t\in I_{r^2}}\|h(\cdot,t)\|_{C^{2,\theta}(B_r(y_0))}\cdot \sup_{t\in I}\|\phi(\cdot,t)\|_{C^{0,\theta}(B_r(y_0,0))}\\
	&\lesssim\sup_{t\in I_{r^2}}\|\phi(\cdot,t)\|^2_{C^{2,\theta}(B_r(y_0,0))}\\
	&\lesssim \|\phi\|^2_{C^{2,\theta}(B_r(y_0,0)\times I_{r^2})}.
	\end{align*}
By the error estimates \eqref{LaplacianError}, \eqref{MetricError} and \eqref{PhiBoundsh}
	\begin{align}\label{HolderLaplacianForPhi}
	&\sup_{t\in I_{r^2}}\left\|\int [\Delta_{\Gamma_{0,t}}-\Delta_{\Gamma_{z,t}}]\phi(\cdot,z,t)(g^*)'(\cdot,z,t) dz\right\|_{C^\theta(B^n_r(y_0))}\\\nonumber
	&\lesssim \left\|[\Delta_{\Gamma_{0,t}}-\Delta_{\Gamma_{z,t}}]\phi(\cdot,z,t)\right\|_{C^\theta(B^n_r(y_0))}+\|\phi(\cdot,t)\|_{C^{2,\theta}(B_r(y_0,0))}\cdot \|(g^*)'(\cdot,t)\|_{C^{\theta}(B_r(y_0,0))}\\\nonumber
	&=\lesssim\sup_{v,w\in B_r^n(y_0)}\frac{[\Delta_{\Gamma_{0,t}}-\Delta_{\Gamma_{z,t}}]\phi(v,z,t)-[\Delta_{\Gamma_{0,t}}-\Delta_{\Gamma_{z,t}}]\phi(w,z,t)}{|v-w|^\theta}\\\nonumber
	&+\|\phi(\cdot,t)\|_{C^{2,\theta}(B_r(y_0,0))}\cdot \|h(\cdot,t)\|_{C^{2,\theta}(B_r(y_0))}\\\nonumber
	&\lesssim\sup_{v,w\in B_r^n(y_0)}\left|\frac{     \sum_{i,j=1}^n(g^{ij}_{\Gamma_{z,t}}(v)-g^{ij}_{\Gamma_{0,t}}(v) ) \frac{\partial^2\phi}{\partial v_iv_j} - \sum_{i,j=1}^n(g^{ij}_{\Gamma_{z,t}}(w)-g^{ij}_{\Gamma_{0,t}}(w) ) \frac{\partial^2\phi}{\partial w_iw_j}   }{|v-w|^\theta}\right|\\\nonumber
	&+\sup_{v,w\in B_r^n(y_0)}\left|\frac{     \sum_{i=1}^n(b^i_{\Gamma_{z,t}}(v)-b^i_{\Gamma_{0,t}}(v) ) \frac{\partial\phi}{\partial v_i} -  \sum_{i=1}^n(b^i_{\Gamma_{z,t}}(w)-b^i_{\Gamma_{0,t}}(w) ) \frac{\partial\phi}{\partial w_i}   }{|v-w|^\theta}\right|\\\nonumber
	&+\sup_{t\in I_{r^2}}\|\phi(\cdot,t)\|^2_{C^{2,\theta}(B_r(y_0,0))}\\\nonumber
	&\lesssim\sup_{v,w\in B_r^n(y_0)}\left|\frac{     \sum_{i,j=1}^n \left(g^{ij}_{\Gamma_{z,t}}(v)-g^{ij}_{\Gamma_{0,t}}(v) \right) \left(\frac{\partial^2\phi}{\partial v_iv_j} - \frac{\partial^2\phi}{\partial w_iw_j     }\right)  }{|v-w|^\theta}\right|\\\nonumber
	&+\sup_{v,w\in B_r^n(y_0)}\left|\frac{     \sum_{i,j=1}^n \left[\left(g^{ij}_{\Gamma_{z,t}}(v)-g^{ij}_{\Gamma_{z,t}}(w) \right)-\left(g^{ij}_{\Gamma_{0,t}}(v)-g^{ij}_{\Gamma_{0,t}}(w) \right) \right] \frac{\partial^2\phi}{\partial w_iw_j     } }{|v-w|^\theta}\right|\\\nonumber
	&+\sup_{v,w\in B_r^n(y_0)}\left|\frac{     \sum_{i=1}^n(b^i_{\Gamma_{z,t}}(v)-b^i_{\Gamma_{0,t}}(v) ) \left(\frac{\partial\phi}{\partial v_i} - \frac{\partial\phi}{\partial w_i} \right)   }{|v-w|^\theta}\right|\\\nonumber
	&+\sup_{v,w\in B_r^n(y_0)}\left|\frac{     \sum_{i=1}^n\left[(b^i_{\Gamma_{z,t}}(v)-b^i_{\Gamma_{z,t}}(w) )- (b^i_{\Gamma_{0,t}}(v)-b^i_{\Gamma_{0,t}}(w) )\right]\frac{\partial\phi}{\partial w_i}   }{|v-w|^\theta}\right|\\\nonumber
	&+\sup_{t\in I_{r^2}}\|\phi(\cdot,t)\|^2_{C^{2,\theta}(B_r(y_0,0))}\\\nonumber
	&\lesssim\e\sup_{t\in I_{r^2}}\|\phi(\cdot,t)\|_{C^{2,\theta}(B_r(y_0,0))}+\e\sup_{t\in I_{r^2}}\|\phi(\cdot,t)\|_{C^{2}(B_r(y_0,0))}+\e\sup_{t\in I_{r^2}}\|\phi(\cdot,t)\|_{C^{1,\theta}(B_r(y_0,0))}\\\nonumber
	&+\e\sup_{t\in I_{r^2}}\|\phi(\cdot,t)\|_{C^{1}(B_r(y_0,0))}+\sup_{t\in I_{r^2}}\|\phi(\cdot,t)\|^2_{C^{2,\theta}(B_r(y_0,0))}\\\nonumber
	&\lesssim \e^2+\|\phi\|^2_{C^{2,\theta}(B_r(y_0,0)\times I)},
	\end{align}
	where $b^i_{\Gamma_{z,t}}(y)=\frac{1}{2}\sum_{j=1}^ng^{ij}_{\Gamma_{z,t}}(y)\frac{\partial}{\partial y_j}\log\det\left((g_{ij})_{\Gamma_{z,t}}(y)\right)$ above.
	
We also estimate the term
	\begin{align*}
	&\sup_{t\in I_{r^2}}\left\|\left(\left\langle \frac{\partial X_{\Gamma_t}}{\partial t},\nu_{\Gamma_t}\right\rangle+\frac{\partial h}{\partial t}\right)(\cdot,t)\left(\int\phi(\cdot,z,t) (g^*)'' (\cdot,z,t)dz\right)\right\|_{C^\theta(B^n_r(y_0))}\\
	&\lesssim \sup_{t\in I_{r^2}}\left\|\left(\left\langle \frac{\partial X_{\Gamma_t}}{\partial t},\nu_{\Gamma_t}\right\rangle+\frac{\partial h}{\partial t}\right)(\cdot,t)\right\|^2_{C^{\theta}(B^n_r(y_0))}+\sup_{t\in I_{r^2}}\left\|\int\phi (\cdot,z, t)(g^*)''(\cdot,z,t) dz\right\|^2_{C^{\theta}(B^n_r(y_0))}\\
	&\lesssim \sup_{t\in I_{r^2}}\left\|\left\langle \frac{\partial X_{\Gamma_t}}{\partial t},\nu_{\Gamma_t}\right\rangle(\cdot, t)\right\|^2_{C^{\theta}(B_r(y_0))}+\sup_{t\in I}\left\|\frac{\partial h}{\partial t}(\cdot, t)\right\|^2_{C^{\theta}(B^n_r(y_0))}\\
	&+\sup_{t\in I_{r^2}}\left\|\int\phi (\cdot,z, t)(g^*)''(\cdot,z,t) dz\right\|^2_{C^{\theta}(B^n_r(y_0))}\\
	&\lesssim \e^2+\|h\|^2_{C^{2,\theta}(B^n_r(y_0)\times I_{r^2})}+\sup_{t\in I}\|\phi(\cdot, t)\|^2_{C^{\theta}(B_r(y_0,0))}\\
	&\lesssim \e^2+\|\phi\|^2_{C^{2,\theta}(B_r(y_0,0)\times I_{r^2})},
	\end{align*}
where the bound
	\begin{align*}
	&\sup_{t\in I_{r^2}}\left\|\left\langle \frac{\partial X_{\Gamma_t}}{\partial t},\nu_{\Gamma_t}\right\rangle(\cdot, t)\right\|^2_{C^{\theta}(B^n_r(y_0))}\\
	&=\sup_{t\in I_{r^2}}\left\|\frac{\frac{\partial u}{\partial t}}{|\nabla u|}(\cdot, t)\right\|^2_{C^{\theta}(B^n_r(y_0))}\\
	&\lesssim \sup_{t\in I_{r^2}}\left\| \frac{\partial u}{\partial t}(\cdot, t)\right\|^2_{C^{\theta}(B^n_r(y_0))}+ \sup_{t\in I_{r^2}}\left|\frac{\partial u}{\partial t}\right| ^2\|u\|^2_{C^{2,\theta}(B^n_{r'}(y_0)\times I_{(r')^2})}    \\
	&\lesssim \sup_{t\in I}\left\|\frac{\partial}{\partial t}u(\cdot,0,t)\right\|^2_{C^{\theta}(B^n_r(y_0))}\\
	&\lesssim \e^2+\|\phi\|^2_{C^{2,\theta}(B^n_{r'}(y_0)\times I_{(r')^2})}
	\end{align*}
follows by Lemma \ref{LemmaEstimateDifferenceToFlatSolution} and the uniform lower gradient bound in Theorem \ref{GradientUpperLowerBound}.

Similarly, using \eqref{PhiBoundsh}, we get
	\begin{align*}
	&\sup_{t\in I_{r^2}}\left\|\int \left\langle\frac{\partial\phi}{\partial y_i}(\cdot,z,t), \frac{\partial h}{\partial y_j}(\cdot,t)\right\rangle_{g_{\Gamma_{0,t}}}(g^*)''(\cdot,z,t)dz\right\|_{C^\theta(B^n_r(y_0))}\\
	&\lesssim\sup_{t\in I_{r^2}} \|h(\cdot,t)\|_{C^{1,\theta}(B_r(y))}\cdot\|\phi(\cdot,t)\|_{C^{1,\theta}(B_r(y_0,0))}\\
	&\lesssim\sup_{t\in I_{r^2}}\|\phi(\cdot,t)\|^2_{C^{1,\theta}(B_r(y_0,0))}\\
	&\lesssim\|\phi\|^2_{C^{2,\theta}(B_r(y_0,0)\times I_{r^2})},
	\end{align*}
and
	\begin{align*}
	&\sup_{t\in I_{r^2}}\left\||\nabla_{\Gamma_{0,t}}h(\cdot,t)|^2\int\phi (\cdot,t) (g^*)'''(\cdot,z,t)dz\right\|_{C^\theta(B^n_r(y_0))}\\
	&\lesssim\sup_{t\in I_{r^2}}\|h(\cdot,t)\|^2_{C^{1,\theta}(B^n_r(y_0))}\cdot\sup_{t\in I_{r^2}}\|\phi(\cdot,t)\|_{C^{\theta}(B_r(y_0,0))}\\
	&\lesssim\|\phi\|^2_{C^{2,\theta}(B_r(y_0,0)\times I_{r^2})}.
	\end{align*}
Using \eqref{CurvatureError}, \eqref{MetricError}, Lemma \ref{LemmaErrorInZ} and integration by parts, we have
	\begin{align*}
	&\sup_{t\in I_{r^2}}\left\|\int (g^*)'(\cdot,z,t)H_{\Gamma_{z,t}}\partial_z\phi(\cdot,z,t) dz\right\|_{C^\theta(B^n_r(y_0))}\\
	&\leq C\sup_{t\in I_{r^2}}\left\|\int (g^*)'(\cdot,z,t)\partial_zH_{\Gamma_{z,t}}\phi(\cdot,z,t) + \partial_z(g^*)'(\cdot,z,t)H_{\Gamma_{z,t}}\phi(\cdot,z,t)dz\right\|_{C^\theta(B^n_r(y_0))}\\
	&\lesssim\e\sup_{t\in I_{r^2}}\|\phi(\cdot,t)\|_{C^{2,\theta}(B_r(y_0,0))}\\
	&\lesssim\e^2+\|\phi\|^2_{C^{2,\theta}(B_r(y_0,0)\times I_{r^2})}.
	\end{align*}
By the uniform bounds on derivatives of $\bar g$, \eqref{MetricError} and Lemma \ref{PhiBoundsh}
	\begin{align*}
	&\sup_{t\in I_{r^2}}\left\|\int (g^*)'(\cdot,z,t)(g^*)''(\cdot,z,t)|\nabla_{\Gamma_{z,t}} h(\cdot,t)|^2 dz\right\|_{C^\theta(B^n_r(y_0))}\\
	&\lesssim\|\sup_{t\in I}h(\cdot,t)\|^2_{C^{1,\theta}(B^n_r(y_0))}\\
	&\lesssim\|\phi\|^2_{C^{2,\theta}(B_r(y_0,0)\times I_{r^2})}.
	\end{align*}
Again, similar to the estimate carried out in \eqref{HolderLaplacianForPhi}, by \eqref{LaplacianError} and \eqref{CurvatureError}, we have
	\begin{align*}
	&\sup_{t\in I_{r^2}}\left\|\int[\Delta_{\Gamma_{z,t}}h(\cdot,t)-\Delta_{\Gamma_{0,t}}h(\cdot,t)] [(g^*)'(\cdot,z,t)]^2 dz\right\|_{C^\theta(B^n_r(y_0))}\\
	&\lesssim\e\sup_{t\in I_{r^2}}\|h(\cdot,t)\|_{C^{2,\theta}(B^n_r(y_0))}\\
	&\lesssim\e^2+\|h\|^2_{C^{2,\theta}(B^n_r(y_0)\times I_{r^2})}\\
	&\lesssim\e^2+\|\phi\|^2_{C^{2,\theta}(B_r(y_0,0)\times I_{r^2})},
	\end{align*}
and
	\begin{align*}
	\sup_{t\in I_{r^2}}\left\|\int[H_{\Gamma_{z,t}}-H_{\Gamma_{0,t}}] [(g^*)'(\cdot,z,t)']^2 dz\right\|_{C^\theta(B^n_r(y_0))}= O(\e^2).
	\end{align*}
Combining all these above estimates and
	\begin{align*}
	&\left\|\left[\frac{\partial h}{\partial t}(\cdot,t)-\Delta_{\Gamma_t}h(\cdot,t)+H_{\Gamma_t}+\left\langle \frac{\partial X_{\Gamma_t}}{\partial t},\nu_{\Gamma_t}\right\rangle\right]\left(\int (g^*)'^2(\cdot,z,t) dz\right)\right\|_{C^{0,\theta}(B^n_r(y_0))}\\
	&= \alpha\left\|\frac{\partial h}{\partial t}(\cdot,t)-\Delta_{\Gamma_t}h(\cdot,t)+H_{\Gamma_t}+\left\langle \frac{\partial X_{\Gamma_t}}{\partial t},\nu_{\Gamma_t}\right\rangle\right\|_{C^{0,\theta}(B^n_r(y_0))}+O(\e^2),
	\end{align*}
we obtain from \eqref{EquationForHolder}
	\begin{align}\label{SpaceHolder}
	\begin{split}
&\sup_{t\in I_{r^2}}\left\|\left[\frac{\partial h}{\partial t}(\cdot,t)-\Delta_{\Gamma_t}h(\cdot,t)+H_{\Gamma_t}+\left\langle \frac{\partial X_{\Gamma_t}}{\partial t},\nu_{\Gamma_t}\right\rangle\right]\right\|_{C^{0,\theta}(B^n_r(y_0))}\\
	&\lesssim\e^2+\|\phi\|^2_{C^{2,\theta}( B_r(y_0,0)\times I_{r^2})}.
\end{split}
	\end{align}

\subsubsection{H\"older estimates in time}
Again we estimate the H\"older norm in time term by term and use the notation in \eqref{Notations} for domains. By \eqref{LaplacianError},
	\begin{align*}
	&\sup_{y\in B^n_r(y_0)}\left\|\int[\Delta_{\Gamma_{0,\cdot}}-\Delta_{\Gamma_{z,\cdot}}]\phi (y,z,\cdot) (g^*)' (y,z,\cdot)dz\right\|_{C^\frac{\theta}{2}(I_{r^2})}\\
	&\lesssim\sup_{y\in B^n_r(y_0)}\|[\Delta_{\Gamma_{0,\cdot}}-\Delta_{\Gamma_{z,\cdot}}]\phi (y,0,\cdot)\|_{C^\frac{\theta}{2}(I_{r^2})}\\
	&+\sup_{(y,z)\in B_r(y_0,0)}\left|[\Delta_{\Gamma_{0,\cdot}}-\Delta_{\Gamma_{z,\cdot}}]\phi (y,0,\cdot)\right|\|h(y,\cdot)\|_{C^\frac{\theta}{2}(I_{r^2})} \\
	&\lesssim\e\|\phi\|_{C^{2,\theta}(B_{r'}(y_0,0)\times I_{(r')^2})}\\
	&\lesssim\e^2+\|\phi\|^2_{C^{2,\theta}(B_{r'}(y_0,0)\times I_{(r')^2})}.
	\end{align*}
By \eqref{HolderSmallnessRescaled} and \eqref{PhiBoundsh}
	\begin{align*}
	&\sup_{y\in B_r^n(y_0)}\left\| \left(\left\langle \frac{\partial X_{\Gamma_{0,\cdot}}}{\partial t},\nu_{\Gamma_{0,\cdot}}\right\rangle+\frac{\partial h}{\partial t}\right)(y,\cdot)\int\phi(y,z,\cdot) (g^*)'' (y,z,\cdot)dz \right\|_{C^\frac{\theta}{2}(I_{r^2})}\\
	&\lesssim\sup_{y\in B_r^n(y_0)}\left\|\left(\left\langle \frac{\partial X_{\Gamma_{0,\cdot}}}{\partial t},\nu_{\Gamma_{0,\cdot}}\right\rangle+\frac{\partial h}{\partial t}\right)(y,\cdot)\right\|^2_{C^\frac{\theta}{2}(I_{r^2})}+\sup_{y\in B^n_r(y_0)}\left\|\int\phi(y,z,\cdot) (g^*)''(y,z,\cdot) dz \right\|^2_{C^\frac{\theta}{2}(I_{r^2})}\\
	&\lesssim\sup_{y\in B_r^n(y_0)}\left\|\left\langle \frac{\partial X_{\Gamma_{0,\cdot}}}{\partial t},\nu_{\Gamma_{0,\cdot}}\right\rangle(y,\cdot)\right\|^2_{C^\frac{\theta}{2}(I_{r^2})}+\sup_{y\in B_r^n(y_0)}\left\|\frac{\partial h}{\partial t}(y,\cdot)\right\|^2_{C^\frac{\theta}{2}(I_{r^2})}\\
	&+\sup_{(y,z)\in B_r(y_0,0)}\|\phi(y,z,\cdot)\|^2_{C^\frac{\theta}{2}(I_{r^2})}\\
	&\lesssim\sup_{y\in B_r^n(y_0)}\left\|\frac{\partial u}{\partial t}(y,0,\cdot)\right\|^2_{C^\frac{\theta}{2}(I_{r^2})}+\sup_{y\in B_r^n(y_0)}\left\|\frac{\partial h}{\partial t}(y,\cdot)\right\|^2_{C^\frac{\theta}{2}(I_{r^2})}\\
	&+\sup_{(y,z)\in B_r(y_0,0)}\|\phi(y,z,\cdot)\|^2_{C^\frac{\theta}{2}(I_{r^2})}\\
	&\lesssim \e^2+\|\phi\|^2_{C^{2,\theta}(B_{r'}(y_0,0)\times I_{(r')^2})},
	\end{align*}
where we used Lemma \ref{LemmaEstimateDifferenceToFlatSolution} to bound $\left\|\frac{\partial u}{\partial t}(y,\cdot)\right\|^2_{C^\frac{\theta}{2}(I_{r^2})}=\left\|\frac{\partial (u-\tilde g)}{\partial t}(y,\cdot)\right\|^2_{C^\frac{\theta}{2}(I_{r^2})}$ in the second last line.

By \eqref{PhiBoundsh} again,
	\begin{align*}
	&\sup_{y\in B_r^n(y_0)}\left\|\Delta_{\Gamma_{0,\cdot}}h(y,\cdot)\int\phi(y,z,\cdot) (g^*)''(y,z,\cdot) dz\right\|_{C^\frac{\theta}{2}(I_{r^2})}\\
	&\lesssim\sup_{y\in B_r^n(y_0)}\|\Delta_{\Gamma_{0,\cdot}}h\|^2_{C^\frac{\theta}{2}(I_{r^2})}+\sup_{y\in B_r^n(y_0)}\left\|\int\phi(y,z,\cdot) (g^*)''(y,z,\cdot) dz\right\|^2_{C^\frac{\theta}{2}(I_{r^2})}\\
	&\lesssim \|h\|^2_{C^{2,\theta}(B_r(y_0)\times I_{r^2})}+\|\phi\|^2_{C^{2,\theta}(B_r(y_0,0)\times I_{r^2})}\\
	&\lesssim \|\phi\|^2_{C^{2,\theta}(B_r(y_0,0)\times I_{r^2})},
	\end{align*}
	\begin{align*}
	&\sup_{y\in B_r^n(y_0)}\left\|\int \left\langle\frac{\partial\phi}{\partial y_i}(y,z,\cdot), \frac{\partial h}{\partial y_j}(y,\cdot)\right\rangle_{g_{\Gamma_{0,\cdot}}}(g^*)''(y,z,\cdot) dz\right\|_{C^\frac{\theta}{2}(I_{r^2})}\\
	&\lesssim \|\phi\|_{C^{1,\theta}(B_r(y_0,0)\times I_{r^2})}\cdot \|h\|_{C^{1,\theta}(B_r(y_0)\times I_{r^2})}\\
	&\lesssim\|\phi\|^2_{C^{2,\theta}(B_r(y_0,0)\times I_{r^2})},
	\end{align*}
	\begin{align*}
	&\sup_{y\in B_r^n(y_0)}\left\| |\nabla_{\Gamma_{0,\cdot}}h(y,\cdot)|\int\phi(y,z,\cdot) (g^*)''' (y,z,\cdot)dz\right\|_{C^\frac{\theta}{2}(I_{r^2})}\\
	&\lesssim\|h\|_{C^{1,\theta}(B_r^n(y_0)\times I_{r^2})})\cdot \|\phi\|_{C^{1,\theta}B_r(y_0,0)\times I_{r^2})}\\
	&\lesssim\|\phi\|^2_{C^{2,\theta}(B_r(y_0,0)\times I_{r^2})},
	\end{align*}
and
	\begin{align*}
	&\sup_{y\in B_r^n(y_0)}\left\|\int (g^*)'H_{\Gamma_{z,t}}\partial_z\phi(y,z,\cdot) dz\right\|_{C^\frac{\theta}{2}(I_{r^2})}\\
	&\lesssim\e\cdot\|\phi\|_{C^{1,\theta}B_r(y_0,0)\times I_{r^2})}\\
	&\lesssim\e^2+\|\phi \|^2_{C^{2,\theta}(B_r(y_0,0)\times I_{r^2})}.
	\end{align*}
By the uniform smallness of deviation in the $z$ coordinate \eqref{CurvatureError} and \eqref{CurvatureDerivativeError}
	\begin{align*}
	&\sup_{y\in B_r^n(y_0)}\left\|\int[H_{\Gamma_{z,t}}-H_{\Gamma_{0,t}}](g^*)'^2(y,z,\cdot) dz\right\|_{C^\frac{\theta}{2}(I_{r^2})}\\
	&\lesssim \sup_{y\in B_r^n(y_0), t\in I_{(r')^2}}\left|\int\frac{\partial}{\partial t}\left[(H_{\Gamma_{z,t}}-H_{\Gamma_{0,t}})(g^*)'^2(y,z,\cdot)\right] dz\right|\\
	&\lesssim \e^2+\|\phi \|^2_{C^{2,\theta}(B_{r'}(y_0,0)\times I_{(r')^2})},
	\end{align*}
and by \eqref{LaplacianHolderError}, \eqref{PhiBoundsh}
	\begin{align*}
	&\sup_{y\in B_r^n(y_0)}\left\|\int[\Delta_{\Gamma_{z,\cdot}}h(y,\cdot)-\Delta_{\Gamma_{0,\cdot}}h(y,\cdot)](g^*)'^2(y,z,\cdot) dz\right\|_{C^\frac{\theta}{2}(I_{r^2})}\\
	&\lesssim\e^2+\|\phi \|^2_{C^{2,\theta}(B_{r'}(y_0,0)\times I_{(r')^2})},
	\end{align*}
	\begin{align*}
	&\sup_{y\in B_r^n(y_0)}\left\|\int g'g''|\nabla_{z,t}h(y,\cdot)|^2 dz\right\|_{C^\frac{\theta}{2}(I_{r^2})}\\
	&\lesssim\e^2+\|\phi \|^2_{C^{2,\theta}(B_{r'}(y_0,0)\times I_{(r')^2})}.
	\end{align*}
And thus similar to \eqref{SpaceHolder}, we obtain
	\begin{align}\label{TimeHolder}
	\begin{split}
&\sup_{y\in B_r^n(y_0)}\left\| \left[\frac{\partial h}{\partial t}-\Delta_{\Gamma_{0,\cdot}}h+H_{\Gamma_{0,\cdot}}+\left\langle \frac{\partial X_{\Gamma_{0,\cdot}}}{\partial t},\nu_{\Gamma_{0,\cdot}}\right\rangle\right](y,\cdot)\right\|_{C^{\frac{\theta}{2}}(I_{r^2})}\\
	&\lesssim\e^2+\|\phi\|^2_{C^{2,\theta}(B_{r'}(y_0,0)\times I_{(r')^2})}.
\end{split}
	\end{align}



From \eqref{SpaceHolder}, \eqref{TimeHolder}, we put back the omitted superscript $\e$ and obtain H\"older estimates for the term II+IV in space-time
	\begin{align}\label{HolderEstimate}
	\begin{split}
&\left\|\left[\frac{\partial h}{\partial t}-\Delta_{\Gamma^\e_{0,\cdot}}h+H_{\Gamma^\e_{0,\cdot}}+\left\langle \frac{\partial X_{\Gamma^\e_{0,\cdot}}}{\partial t},\nu_{\Gamma^\e_{0,\cdot}}\right\rangle\right]\right\|_{C^{0,\theta}(B_r^n(y_0)\times I_{r^2})}\\
	&\lesssim \e^2+\|\phi^\e\|^2_{C^{2,\theta}(B_{r'}(y_0,0)\times I_{(r')^2})}\\
	&\lesssim\e^2+\sigma\|\phi^\e\|_{C^{2,\theta}(B_{r'}(y_0,0)\times I_{(r')^2})}
\end{split}
	\end{align}
where $\sigma=o(1)$ is a small coefficient since the norm of $\phi^\e$ is small. This follows from the assumption of a uniform bound on $\mathcal A$ for $u_\e$ in Theorem \ref{ImprovementRegularity}. The rescaled solution sequence $u^\e$ is a solution of \eqref{AC}. As shown in Section \ref{ApproximateSolution}, all the derivatives of $\phi$ converge to $0$ uniformly, which implies that $|\phi^\e|_{C^{2,\theta}}=o(1)$ for small $\e$. This will be used later in an iteration argument.
\section{Parabolic Schauder estimates for $\phi$ and regularity of the level sets, the proof of main theorems}\label{ProofMainTheorem}
Rewriting the equation \eqref{Difference}, we get
	\begin{align*}
	\left(\frac{\partial}{\partial t}-\Delta\right)\phi^\e+W''(g^{\e,*})\phi=&(g^{\e,*})'\left(\frac{\partial h}{\partial t}-\Delta_{\Gamma^\e_{0,t}}h+H_{\Gamma^\e_{0,t}}+\left\langle \frac{\partial X_{\Gamma^\e_t}}{\partial t},\nu_{\Gamma^\e_t}\right\rangle\right)\\&+(g^{\e,*})''|\nabla h|^2+\bar\eta+\mathcal R(\phi),
	\end{align*}
	where the remainder term $\mathcal R$ is of order $O(|\phi|^2)$.
	
By applying standard parabolic Schauder estimates (see chapter 4 of \cite{lieberman1996second} for a reference) to the above equation, we get by \eqref{PhiBoundsh}
	\begin{align*}
	&\|\phi^\e\|_{C^{2,\theta}(D_r(y)\times I_{r^2})}\\
	&\lesssim\left\|\frac{\partial h}{\partial t}-\Delta_{\Gamma^\e_{0,t}}h+H_{\Gamma^\e_{0,t}}+C\left\langle \frac{\partial X_{\Gamma^\e_t}}{\partial t},\nu_{\Gamma^\e_t}\right\rangle\right\|_{C^{0,\theta}(B_{r'}(y_0)\times I_{(r')^2})}+\|h\|^2_{C^{2,\theta}(B_{r'}(y_0)\times I_{(r')^2})}+\e^2\\
	&\lesssim \left\|\frac{\partial h}{\partial t}-\Delta_{\Gamma^\e_{0,t}}h+H_{\Gamma^\e_{0,t}}+\left\langle \frac{\partial X_{\Gamma^\e_t}}{\partial t},\nu_{\Gamma^\e_t}\right\rangle\right\|_{C^{0,\theta}(B_{r'}(y_0)\times I_{(r')^2})}+\sigma\|\phi^\e\|_{C^{2,\theta}(B_{r'}(y_0,0)\times I_{(r')^2})}+\e^2
	\end{align*}
where $\sigma<1$ is a small constant.

Combining this with the estimate \eqref{HolderEstimate}, we get
	\begin{align*}
	&\left\|\frac{\partial h}{\partial t}-\Delta_{\Gamma^\e_{0,t}}h+H_{\Gamma^\e_{0,t}}+\left\langle \frac{\partial X_{\Gamma^\e_t}}{\partial t},\nu_{\Gamma^\e_t}\right\rangle\right\|_{C^{0,\theta}(B_r(y_0)\times I_{r^2})}+\|\phi^\e\|_{C^{2,\theta}(B_r(y_0,0)\times I_{r^2})}\\
	&\lesssim \e^2+\sigma\left(\left\|\frac{\partial h}{\partial t}-\Delta_{\Gamma^\e_{0,t}}h+H_{\Gamma^\e_{0,t}}+\left\langle \frac{\partial X_{\Gamma^\e_t}}{\partial t},\nu_{\Gamma^\e_t}\right\rangle\right\|_{C^{0,\theta}(B_{r'}(y_0)\times I_{(r')^2})}+\|\phi^\e\|_{C^{2,\theta}(B_{r'}(y_0,0)\times I_{(r')^2})}\right).
	\end{align*}
We can then iterate this inequality $K>|\log\e|$ times (so that $\sigma^K=O(\e)$), and obtain
	\begin{align*}
	&\left\|\frac{\partial h}{\partial t}-\Delta_{\Gamma^\e_{0,t}}h+H_{\Gamma^\e_{0,t}}+\left\langle \frac{\partial X_{\Gamma^\e_t}}{\partial t},\nu_{\Gamma^\e_t}\right\rangle\right\|_{C^{0,\theta}(B_r(y_0)\times I_{r^2})}+\|\phi^\e\|_{C^{2,\theta}(B_r(y_0,0)\times I_{r^2})}\\
	&= O(\e^2|\log\e|)= O(\e^{2-\delta}),
	\end{align*}
for any $\delta>0$.

Thus by \eqref{PhiBoundsh}, we have
\begin{equation*}
\left\|H_{\Gamma^\e_{0,t}}+\left\langle \frac{\partial X_{\Gamma^\e_{0,t}}}{\partial t},\nu_{\Gamma^\e_{0,t}}\right\rangle\right\|_{C^{0,\theta}(B_r(y_0)\times I_{r^2})}= O(\e^{2-\delta}), \forall\delta>0.
\end{equation*}
After rescaling back to the original scale we have
\begin{equation*}
\left\|H_{\Gamma_{\e,0,t}}+\left\langle \frac{\partial X_{\Gamma_{\e,0,t}}}{\partial t},\nu_{\Gamma_{\e,0,t}}\right\rangle\right\|_{C^{0,\theta}(B_r(y_0)\times I_{r^2})}= O(\e^{1-\theta})\leq C_\theta,
\end{equation*}
for any $\theta\in(0,1)$.

By the curvature bound $|\mathcal A_\e|\leq C$, the nodal sets $\Gamma_{\e,0,t}$ can be represented by local graphs over $\Gamma_{\e,0,t}\cap \{B_r\times I_{r^2}\}=\{x_{n+1}=f_\e\}\cap \{B_r\times I_{r^2}\}$ with bounded gradient $|\nabla f_\e|\leq \tilde C$, so the parabolic Schauder estimates in Lemma \ref{ParabolicSchauder} gives
	\begin{align}\label{UniformEstimate}
	\|f_\e\|_{C^{2,\theta}(B_r(y_0)\times I)}\leq C_\theta,
	\end{align}
for any $\theta\in(0,1)$

\begin{proof}[Proof of Theorem \ref{ImprovementRegularity}]
We first perform a parabolic translation in space-time for $u_\e$ so that $u_\e(0,0)=0$ as in Remark \ref{VanishingAtZero}, which we still denote by $u_\e$. By applying the curvature bound and the entropy bound $\lambda<2\alpha-\delta$, we obtain a smooth limit flow with multiplicity $1$. The proof of Theorem \ref{ImprovementRegularity} is essentially completed through the last equation \eqref{UniformEstimate}. This equation provides a uniform $C^{2,\theta}$ estimate, which is obtained by combining a uniform enhanced second fundamental form bound with the aforementioned entropy bound. Consequently, we can obtain a uniform $C^{0,\theta}$ norm for the curvature of the level sets.
\end{proof}

\begin{proof}[Proof of Theorem \ref{Graphical}]
This argument is similar to the proof of Corollary 1.2 in \cite{Wang2019} where the conditions in Theorem \ref{Graphical} imply the conditions (uniform enhanced second fundamental form bounds) in Theorem \ref{ImprovementRegularity} by a blow up argument.
\end{proof}

\section{Proof of curvature estimates}\label{ProofCurvature}
In this section we prove the a priori bound on enhanced second fundamental forms for low entropy Allen--Cahn flows.
\begin{proof}[Proof of Corollary \ref{CurvatureEstimates}]
Again we argue by contradiction. Let us assume there exists a sequence $\e_i\rightarrow0$ and a sequence of solutions $u_{\e_k}$ to the equation \eqref{EAC} in $B_r(0)\times[-r^2,r^2]\subset\mathbb R^2\times\mathbb R$ with $\e=\e_k$ such that $u(0,0)=0$. After passing to a subsequence, we can obtain a limit Brakke flow $d\mu_{\e_k,t}=\left[\frac{\e_k|\nabla u_{\e_k}(\cdot,t)|^2}{2}+\frac{W(u_{\e_k}(\cdot,t))}{\e_k}\right]dx\rightarrow d\mu_{t}$ using \cite{Ilmanen1993}.
\begin{claim}\label{NonEmptyLimitFlow}
Under the entropy assumption $\lambda(d\mu_\e)\leq2\alpha-\delta$, we have the convergence  $d\mu_{\e_k,t}=\left[\frac{\e_k|\nabla u_{\e_k}(\cdot,t)|^2}{2}+\frac{W(u_{\e_k}(\cdot,t))}{\e_k}\right]dx\rightarrow d\mu_{t}, \forall t\in(-r^2,r^2)$. Moreover $d\mu_t$ is non-empty for every $t\in(-r^2,r^2)$.
\end{claim}
\begin{proof}[Proof of Claim \ref{NonEmptyLimitFlow}]
To prove the claim, we will use a proof by contradiction. Suppose for the sake of contradiction that the claim is false. Then there exists some time $t_{extinct} \in (-r^2, r^2)$ such that the flow is empty for all time $t\in (t_{extinct},r^2)$. So on the compact subset of space-time $B_{2R_0}\times(t_{extinct},r^2)\subset\mathbb R^2\times\mathbb R$, we have uniform convergence (\cite{Ilmanen1993})
	\begin{align}\label{UniformConvergenceAfterExtinction}
	u_{\e_k}\rightarrow 1, \text{ or} -1.
	\end{align}
On the other hand, by the condition on representing phase transition, for any $t_0\in(t_{extinct},r^2)$ and $\e_i>0$, there exists $x_{0,\e_i}\in B_{R_0}$ such that
	\begin{align*}
	u_{\e_k}(x_{0,\e_k},t_0)=0.
	\end{align*}
This gives a contradiction to the uniform convergence \eqref{UniformConvergenceAfterExtinction} and hence the limit Brakke flow is non-empty for any time $t\in\mathbb (-r^2,r^2)$. 

By \cite[5.5]{Ilmanen1993} (c.f. \cite[Proposition 3.4]{Sato2008}), we have convergence for every $t\in\mathbb (-r^2,r^2)$.

The smoothness follows from the entropy condition because the only singularity model in this case is the round shrinking circle for curve shortening flow where the flow becomes extinct.
\end{proof}
We next prove that the gradient is non-zero at $(0,0)$. In fact, there is a lower gradient bound for points in the nodal sets.
\begin{claim}\label{GradientLowerBound2d}
There exists a $c_\delta>0$ depending on the $\delta$ in Corollary \ref{CurvatureEstimates} and independent of $\e$ such that if $u_\e$ satisfies the conditions in Corollary \ref{CurvatureEstimates}, then $\e|\nabla u_\e(0,0)|\geq c_\delta>0$ for all small enough $\e$.
\end{claim}
\begin{proof}[Proof of Claim \ref{GradientLowerBound2d}]
Suppose not, then there exists a sequence of $u_{\e_k}$ satisfying the equation \eqref{EAC} with $\e=\e_k$ and $\e_k\rightarrow0$ such that $u_{\e_k}(0,0)=0$ and $\e_k|\nabla u_{\e_k}(0,0)|\rightarrow0$. we consider a further rescaled sequence 
\begin{align*}
\tilde u^k(x,t)=u_{\e_k}\left(\e_k(x-\tilde x_k),\e_k^2(t-\tilde t_k)^2\right)
\end{align*} so that $\tilde u^k$ satisfies the equation \eqref{AC} for $t\in(-\frac{r^2}{\e_k^2}, \frac{r^2}{\e_k^2})\rightarrow(-\infty,\infty)$ for each $k$.

By parabolic regularity, we have $\tilde u^k\rightarrow\tilde u^\infty$ smoothly on compact subsets after passing to a subsequence. And thus
\begin{align}\label{ContradictingPropertyLimit2d}
\tilde u^\infty(0,0)=0, |\nabla u^\infty(0,0)|=0.
\end{align}



Claim \ref{NonEmptyLimitFlow} and Item (2) in Theorem \ref{CurvatureEstimates}  then guarantees the condition in Theorem \eqref{ParabolicRigidity2D} is satisfied, so $\tilde u^\infty$ is a 1-d solution up to a rotation and translation in space-time. By the property of such 1-d solutions we have $|\nabla\tilde u^\infty(0,0)|\neq0$, which is a contradiction to \eqref{ContradictingPropertyLimit2d}.

\end{proof}
By the Claim \ref{GradientLowerBound2d}, the enhanced second fundamental form is well defined. Now we prove that it has a uniform bound at $(0,0)$ in space-time. Suppose not, then we have
\begin{equation}\label{CurvatureBlowUp}
|\mathcal A(u_{\e_k}(0,0))|=C_k\geq 4\frac{k}{r}\rightarrow\infty.
\end{equation}
There are 2 cases.

Case 1: If $|\mathcal A( u_{\e_k}(x,t))|\leq 8\frac{k}{r}, \forall (x,t)\in \{u_{\e_k}=0\}\cap\left(B_r(0)\times[-r^2,r^2]\setminus B_\frac{r}{2}(0)\times[-\frac{r^2}{4},\frac{r^2}{4}]\right)$, we choose a point $(\tilde x_k,\tilde t_k)\in \{u_{\e_k}=0\}\cap B_\frac{r}{2}(0)\times[-\frac{r^2}{4},\frac{r^2}{4}]$ such that 
       \begin{align*}
       \sup_{(x,t)\in\{u_{\e_k}=0\}\cap B_\frac{r}{2}(0)\times[-\frac{r^2}{4},\frac{r^2}{4}])}|\mathcal A( u_{\e_k}(x,t))| = |\mathcal A( u_{\e_k}(\tilde x_k,\tilde t_k))|.
       \end{align*}
In this case, we have $|\mathcal A( u_{\e_k}(x,t))|\leq 8\frac{k}{r}\leq 2 |\mathcal A( u_{\e_k}(0,0))|\leq 2 |\mathcal A( u_{\e_k}(\tilde x_k,\tilde t_k))|$ for all $(x,t)\in\{u_{\e_k}=0\}\cap B_r(0)\times[-r^2,r^2]$. 

Case 2: If there exists $(x_k,t_k)\in\{u_{\e_k}=0\}\cap \left(B_r(0)\times[-r^2,r^2]\setminus B_\frac{r}{2}(0)\times[-\frac{r^2}{4},\frac{r^2}{4}]\right)$ such that $|\mathcal A( u_{\e_k}(x_k,t_k))|> 8\frac{k}{r}$. Then we have
\begin{align*}
|\mathcal A( u_{\e_k}(x_k,t_k))|\cdot \mathrm{dist}_P\left((x_k,t_k),\{u_{\e_k}=0\}\cap B_\frac{r}{4}(0)\times[-\frac{r^2}{16},\frac{r^2}{16}]\right)>8\frac{k}{r}\cdot\frac{r}{4}=2k.
\end{align*}
By applying the doubling Lemma \cite[Lemma 5.1]{Polacik2007} with 
\begin{align*}
&X=\Sigma=\{u_{\e_k}=0\}\cap \overline{B_r(0)}\times[-r^2,r^2],\\
&D=\{u_{\e_k}=0\}\cap \left(\overline{B_r(0)}\times[-r^2,r^2]\setminus \overline{B_\frac{r}{4}(0)}\times[-\frac{r^2}{16},\frac{r^2}{16}]\right),\\
&\Gamma=\{u_{\e_k}=0\}\cap \overline{B_\frac{r}{4}(0)}\times[-\frac{r^2}{16},\frac{r^2}{16}],
\end{align*}
we can pick a point $(\tilde x_k,\tilde t_k)\in D$ with the property that
\begin{align*}
&|\mathcal A( u_{\e_k}(\tilde x_k,\tilde t_k))|\geq |\mathcal A( u_{\e_k}(x_k,t_k))|,\\
&|\mathcal A( u_{\e_k}(x,t))|\leq 2|\mathcal A( u_{\e_k}(\tilde x_k,\tilde t_k))|, \forall (x,t)\in\{|x-\tilde x_k|\leq\frac{k}{|\mathcal A( u_{\e_k}(\tilde x_k,\tilde t_k))|}, |t-\tilde t_k|\leq \frac{k^2}{|\mathcal A( u_{\e_k}(\tilde x_k,\tilde t_k))|^2}\}.
\end{align*}
We denote by $\tilde C_k=|\mathcal A( u_{\e_k}(\tilde x_k,\tilde t_k))|$ for either cases above and define
        \begin{align*}
	\tilde u_k(x,t)=u_{\e_k}\left(\frac{x-\tilde x_k}{\tilde C_k},\frac{t-\tilde t_k}{\tilde C_k^2}\right).
	\end{align*}
where $u_k$ satisfies \eqref{EAC} with $\e=\tilde \e_k=\e_k \tilde C_k$.
	

In both cases, by scaling, the function $\tilde u_k$ satisfies
	\begin{align*}
	&|\mathcal A(\tilde u_k(0,0))|=1,\\
	&|\mathcal A(\tilde u_k(x,y))|\leq2, \forall (x,y)\in B_{k}\times[-k^2, k^2].
	\end{align*}
Again, we separate into $2$ cases.

If $\lim_{k\rightarrow \infty}\e_k\tilde C_k=0$, then after passing to a subsequence, we get a limit flow $d\tilde\mu_{t}$ defined for all $t\in[-k^2,k^2]\rightarrow(-\infty,\infty)$. Moreover, since $\tilde u_k$ is a parabolic rescaling of $u_{\e_k}$ by $\frac{1}{\tilde C_i}$, the limit $d\tilde\mu_{t}$ of the energy measure of $\tilde u_i$ is a scaling limit of the energy measure $d\mu_{t}$ by $\frac{1}{C_i}$. And we know that $d\mu_{t}$ is smooth by Claim \ref{NonEmptyLimitFlow}, so
	\begin{align*}
	\frac{1}{\alpha}d\tilde\mu_{t}=\mathcal H^2\lfloor P, \forall t\in\mathbb R
	\end{align*}
where $P\subset\mathbb R^2$ is a static hyperplane, which is the tangent plane at a smooth point $(0,0)$. By Theorem \ref{ImprovementRegularity}, we obtain uniform $C^{2,\theta}$ bounds for sufficiently large $i$, implying that the second fundamental form is preserved in the limit. Consequently, if $\mathcal A(\tilde u_k(0,0))$ tends to $1$ as $k\to \infty$, the limit must have non-zero second fundamental form at the origin, which contradicts the fact that the limit is flat.


If $\lim\sup\e_k\tilde C_k\geq\tilde\e_0>0$, we again consider the rescaled sequence as in the proof of Claim \ref{GradientLowerBound2d}
\begin{align*}
\tilde u^k(x,t)=u_{\e_k}\left(\e_k(x-\tilde x_k),\e_k^2(t-\tilde t_k)^2\right)
\end{align*} so that $\tilde u^k$ satisfies the equation \eqref{AC} for each $k$ and that 
\begin{align}\label{CurvatureLowerBoundContradiction}
|\mathcal A(\tilde u^{k}(0,0))|\geq\tilde\e_0>0.
\end{align} 
Again by Claim \ref{NonEmptyLimitFlow}, the condition in Theorem \ref{ParabolicRigidity2D} is satisfied. We obtain after passing to a subsequence a limit solution $\tilde u^\infty$ which is the 1-d heteroclinic solution. The convergence of $\tilde u^k$ to $\tilde u^\infty$ is smooth on compact subsets of space-time by parabolic regularity. This gives a contradiction to the curvature lower bound \eqref{CurvatureLowerBoundContradiction} for large enough $k$ because $\tilde u^\infty$ has flat level sets.

\end{proof}

\section{Appendix}\label{Appendix}
Here we prove some regularity properties (Lemma \ref{LemmaEstimateDifferenceToFlatSolution} and Lemma \ref{LemmaErrorInZ}) of $u$ and its level sets which are standard results from the theory of semilinear parabolic PDEs.

First we prove  an a priori lower bound for the gradients of parabolic Allen--Cahn at points of phase transition given the assumptions of Theorem \ref{ImprovementRegularity}. The enhanced second fundamental form
	\begin{align*}
	\mathcal A(u)=\frac{\sqrt{|\nabla^2 u|^2-|\nabla|\nabla u||^2}}{|\nabla u|}
	\end{align*}
 is well defined only if the gradient does not vanish.
\begin{thm}\label{GradientUpperLowerBound}
For any $r>0$, there exists $\tilde c_r>0$ such that the following holds: if $u_\e$ is a solution of equation \eqref{EAC} in $B_{r}^{n+1}(x_0)\times(t_0-r^2, t_0+r^2)\subset\mathbb R^{n+1}\times\mathbb R$ with the following properties
\begin{enumerate}
\item $u_\e(x_0,t_0)=0$;
\item $\mathcal A_{u_\e}\leq C_0$, where $C_0$ is independent of $\e$.
\end{enumerate}
Then
\begin{equation*}
\frac{1}{\tilde c_r}\leq\e|\nabla u_\e(x_0,t_0)|\leq \tilde c_r
\end{equation*}
in $B_{\e r}^{n+1}(x_0)\times(t_0-\e^2r^2, t_0+\e^2r^2)\subset\mathbb R^{n+1}\times\mathbb R$  for all sufficiently small $\e$.
\end{thm}
\begin{proof}
Without loss of generality we can assume $(x_0,t_0)=(0,0)$.

Suppose the conclusion does not hold, then there exists a sequence of $\e_i\rightarrow0$, a sequence of solutions $u_{\e_i}$ to the equation \eqref{EAC} in $B^{n+1}_{r}\times(-r^2,r^2)\subset\mathbb R^{n+1}\times\mathbb R$ with $\e=\e_i$ and $u_{\e_i}(0,0)=0$, but such that either
\begin{equation*}
\e_i|\nabla u_{\e_i}(x_i,t_i)|\rightarrow0,
\end{equation*}
or
\begin{equation*}
\e_i|\nabla u_{\e_i}(x_i,t_i)|\rightarrow\infty,
\end{equation*}
for some $(x_i,t_i))\in B^{n+1}_{\e_ir}\times(-\e_i^2r^2,\e_i^2r^2)\subset\mathbb R^{n+1}\times\mathbb R$.

By scaling, we obtain a sequence of solutions $u^{\e_i}=u_{\e_i}(\e_i x,\e_i^2t)$ satisfying equation \eqref{AC} in $B_{\frac{r}{\e_i}}\times(-\frac{r^2}{\e_i^2},\frac{r^2}{\e_i^2})\subset\mathbb R^{n+1}\times\mathbb R$ with $u^{\e_i}(0,0)=0$ and
\begin{equation}\label{ContradictionGradient0}
|\nabla u^{\e_i}(\tilde x_i,\tilde t_i)|\rightarrow0,
\end{equation}
or
\begin{equation}\label{ContradictionGradientInfty}
|\nabla u_{\e_i}(\tilde x_i,\tilde t_i)|\rightarrow\infty,
\end{equation}
for some $(\tilde x_i,\tilde t_i))\in B_{r}\times(-r^2,r^2)$.

Notice that by condition (2) the second fundamental forms of the level sets of $u^\e$ are bounded above by $C_0\e\rightarrow0$. Since the level sets are getting flatter and flatter as $\e\rightarrow0$, after passing to a limit, the sequence converges smoothly on any compact subsets to a limit $u^\infty$ of \eqref{AC} that is defined on the whole space-time $\mathbb R^{n+1}\times\mathbb R$. Since the set of critical values are of measure zero by Sard's Theorem, we have that for almost every $s\in[-1,1]$ the level sets $\{u^\infty=s\}$ are flat hyper planes. This tells that $u^\infty$ must be the static $1$-d solution $g$. Moreover, $u^\infty(0,0)=\lim u^{\e_i}(0,0)=0$. By the assumption on entropy bound, $u^\infty$ cannot be the constant solution with value 0, and thus it has to be the 1-d heteroclinic solution with non-zero spatial gradient at $(0,0)$. This gives a contradiction with either \eqref{ContradictionGradient0} or \eqref{ContradictionGradientInfty} for large enough $i$, and thus we must have a gradient upper and lower bound.
\end{proof}
Now we are ready to prove the 2 regularity lemmas of level sets in section 4.
\begin{proof}[Proof of Lemma \ref{LemmaEstimateDifferenceToFlatSolution}]
Note that as described in subsection \ref{Subsection_Fermi}, the curvature bound $|\mathcal A^\e|\leq C_0\e$ gives that the nodal set $\Gamma^\e_{0,-(r')^2}$ of $u^\e$ at time $-(r'^2)$ is $C_0\e$ close in $C^2$ to a flat hyperplane, which without loss of generality is assumed to be $\{x_{n+1}=0\}$.

Both $u^\e$ and $\tilde g$ satisfies equation \eqref{AC}, so their difference satisfies
	\begin{align}\label{eqn_DifferenceToFlatSolution}
	(\frac{\partial}{\partial t}-\Delta)(u^\e-\tilde g)=-\left[W'(u^\e)-W'(\tilde g)\right].
	\end{align}
We have the following $L^\infty$ bound of the right hand side.
\begin{claim}\label{LInfinityRHS}
	\begin{align*}
	\|W'(u^\e)-W'(\tilde g)\|_{L^\infty (B_{r'}(x_0,0)\times I_{(r')^2})}=O(\e)+O\left(\|\phi^\e\|_{C^{2,\theta}(B_{r'}(x_0,0)\times I_{(r')^2})}\right),
	\end{align*}
	where $r''=\frac{r+r'}{2}$.
\end{claim}
\begin{proof}[Proof of Claim \ref{LInfinityRHS}]
The proof requires the use of auxiliary approximation solution $v_\e$ in \eqref{ApproximateSolution_v} and some estimates from \cite{trumper2008relaxation}.

Let $\Sigma_t$ be the limit mean curvature flow with curvature. We note here that by the level sets have locally uniformly bounded curvature, the limit flow must be locally uniformly bounded as $C^{1,\alpha}$ graphs. Since the mean curvature flow equation of graphs are quasi-linear PDEs, standard theory then implies that all the derivatives are locally uniformly bounded. In particular, the limit flow must have curvature locally uniformly bounded by $C_0$. Let $d_{\Sigma_t}:\mathbb R^{n+1}\rightarrow\mathbb R$ denote the signed distance to $\Sigma_t$. The sign of $d_{\Sigma_t}$ is chosen to agree with the sign of the distance $d_{\e,t}$ to the nodal sets $\Gamma_{\e,t}$ as $\e\rightarrow0$.

Following \cite[(43)]{trumper2008relaxation}, we define an approximate solution $v_\e$ as follows:
	\begin{align}\label{ApproximateSolution_v}
	v_\e(x,t)=\chi_{\bar\tau}(d_{\Sigma_t}(x))\bar g(\frac{d_{\Sigma_t}(x)}{\e})+(1-\chi_{\bar\tau}(d_{\Sigma_t}(x)))\frac{d_{\Sigma_t}(x)}{|d_{\Sigma_t}(x)|},
	\end{align}
where $\chi_{\bar\tau}$ is a $C^\infty$ function supported in $[-\bar\tau,\bar\tau]$, $\chi_{\bar\tau}\equiv1$ in $[-\frac{\bar\tau}{2},\frac{\bar\tau}{2}]$, and $-1\leq\chi_{\bar\tau}\leq1$. Here $\tau$ is chosen small enough so that the distance function $d_{\Sigma_t}$ to the flowing hypersurfaces $\Sigma_t$ are well defined.

And in accordance with our notations, we denote its parabolic rescaling by
	\begin{align*}
	v^\e(x,t)=v_\e(\e x,\e^2t).
	\end{align*}
Since all the nodal sets $\Gamma_{\e,t}$ pass through $(0,0)\in\mathbb R^{n+1}\times\mathbb R$, so does the limit limit mean curvature flow $\Sigma_t$. By the uniform curvature bound of $\Sigma_t$, after rescaling, the flows $\frac{\Sigma_t}{\e}$ have curvature bounded by $O(\e)$ and converges to a flat plane which we denote by $\{x_{n+1}=0\}$ without loss of generality. And we must have for each $\e$, the flow must be $O(\e)$ close in $C^2$ to the plane $\{x_{n+1}=0\}$ since they all pass through $(0,0)$. And thus
	\begin{align}\label{DifferenceVandG}
	|v^\e-\tilde g|=|\bar g\circ d_{\frac{\Sigma_t}{\e}}-\bar g\circ\mathrm{dist}_{x_{n+1}=0}|+O(\e)=O(\e).
	\end{align}
At time $-(r'^2)$, since the nodal set $\Gamma^\e_{0,-(r')^2}$ of $u^\e$ is $C_0\e$ close in $C^2$ to $\{x_{n+1}=0\}$, we have
	\begin{align}\label{estimate_u_minus_v_timeslice}
	&|u^\e\left(y,-(r')^2\right)-v^\e\left(y,-(r')^2\right)|\\\nonumber
&\leq|u^\e\left(y,-(r')^2\right)-g^{\e,*}\left(y,-(r')^2\right)|+|g^{\e,*}\left(y,-(r')^2\right)-\bar g(y)|+|\bar g-v^\e|\\\nonumber
&\leq\|\phi^\e\|_{L^\infty(B_{r'}(x_0,0)\times I_{(r')^2})}+\| g\circ\mathrm{dist}_{\Gamma^\e_{0,t}}-g\circ\mathrm{dist}_{\{x_{n+1}=0\}}\|_{L^\infty (B_{r'}(x_0,0)\times I_{(r')^2})}+O(\e)\\\nonumber
&\leq\|\phi^\e\|_{C^{2,\theta}(B_{r'}(x_0,0)\times I_{(r')^2})}+O(\e).
	\end{align}
By the bounds \cite[(45),(46)]{trumper2008relaxation} and parabolic rescaling, we have
	\begin{align}\label{estimate_heat_operator_v}
	\left|(\frac{\partial}{\partial t}-\Delta)v^\e+W'(v^\e)\right|&=\e^2\left|(\frac{\partial}{\partial t}-\Delta)v_\e+\frac{W'(v_\e)}{\e^2}\right|\\\nonumber
&\leq \e^2\cdot C\frac{1}{\e}=O(\e).
	\end{align}
(Note that in the representation in \cite{trumper2008relaxation} the potential function is chosen to be $\frac{W}{2}$ instead of $W$, but this choice does not affect the estimates).

The final estimate we require is a reformulation of \cite[(35)]{trumper2008relaxation}.
\begin{claim}\label{Trumper}
The difference of $u^\e$ and $v^\e$ is bounded by 
	\begin{align}\label{TrumperEstimate}
	&\sup_{B_{r'}(x_0,0)\times I_{(r')^2}}|u^\e-v^\e|(x,t)\\\nonumber
	&\leq C\sup_{B_{r'}(x_0,0)\times I_{(r')^2}}\int_{-(r')^2}^t\int_{\mathbb R^{n+1}}\mathcal H(x-y,t-s)\left| \left(\frac{\partial}{\partial t}-\Delta\right)v^\e+W'(v^\e)\right|(y,x)dyds\\\nonumber
	&+C\sup_{B_{r'}(x_0,0)}\int_{\mathbb R^{n+1}}\mathcal H(x-y,-(r')^2)|u^\e\left(y,-(r')^2\right)-v^\e\left(y,-(r')^2\right)|dy,
	\end{align}
where $\mathcal H$ is the heat kernel in $\mathbb R^{n+1}$.
\end{claim}
\begin{rem}
We note that the main estimate (19) of \cite[Theorem 3.1]{trumper2008relaxation} is independent of $\e$, but requires the condition that $|u^\e-v^\e|\nrightarrow0$. Here in the claim, we instead use the estimate (35) in the proof of Theorem 3.1 in \cite{trumper2008relaxation}, in which the constant $C$ may depend on $\e$. However, this does not require $|u^\e-v^\e|\nrightarrow0$.  It suffices for our purpose because we are applying it to the rescaled functions $u^\e, v^\e$, which are solutions and approximate solutions respectively to the equation \eqref{AC} with $\e=1$.

\end{rem}
Combining \eqref{estimate_u_minus_v_timeslice}, \eqref{estimate_heat_operator_v} and \eqref{TrumperEstimate}, we have
	\begin{align*}
	&\sup_{B_{r'}(x_0,0)\times I_{(r')^2}}|u^\e-v^\e|(x,t)\\
	&\leq C\e\sup_{B_{r'}(x_0,0)\times I_{(r')^2}}\int_{-(r')^2}^t\int_{\mathbb R^{n+1}}dyds\\
	&+C\left(\e+\|\phi^\e\|_{C^{2,\theta}(B_{r'}(x_0,0)\times I_{(r')^2})}\right)\sup_{B_{r'}(x_0,0)\times I_{(r')^2}}\int_{\mathbb R^{n+1}}\mathcal H(x-y,t)dy\\
	&=O(\e)+O\left(\|\phi^\e\|_{C^{2,\theta}(B_{r'}(x_0,0)\times I_{(r')^2})}\right).
	\end{align*}
Now the conclusion of the claim follows by
	\begin{align*}
	&\|W'(u^\e)-W'(\tilde g)\|_{L^\infty (B_{r'}(x_0,0)\times I_{(r')^2})}\\
	&\leq C\|u^\e-\tilde g\|_{L^\infty (B_{r'}(x_0,0)\times I_{(r')^2})}\\
	&\leq C\|u^\e-v^\e\|_{L^\infty (B_{r'}(x_0,0)\times I_{(r')^2})}+C\|v^\e-\tilde g\|_{L^\infty (B_{r'}(x_0,0)\times I_{(r')^2})}\\
	&=O(\e)+O(\|\phi^\e\|_{C^{2,\theta}(B_{r'}(x_0,0)\times I_{(r')^2})}),
	\end{align*}
(where we used that the potential function satisfies $|W''(t)|\leq C$ for $t\in[-1,1]$)
\end{proof}

Using Claim \ref{LInfinityRHS}, we have by parabolic Calderon--Zygmund estimates (c.f. \cite[Theorem 6]{Wang2003}) that in the ball of radius $r''=\frac{r+r'}{2}$.
	\begin{align*}
	\|u^\e-\tilde g\|_{W^{2,p}(B_{r''}(x_0,0)\times I_{(r'')^2})}&\leq\|W'(u^\e)-W'(\tilde g)\|_{L^p (B_{r'}(x_0,0)\times I_{(r')^2})}\\
	&\leq C\|W'(u^\e)-W'(\tilde g)\|_{L^\infty(B_{r'}(x_0,0)\times I_{(r')^2})}\\
	&\leq O(\e)+O(\|\phi^\e\|_{C^{2,\theta}(B_{r'}(x_0,0)\times I_{(r')^2})}),
	\end{align*}
for any $1<p<\infty$. The Sobolev inequalities then give
	\begin{align*}
	\|u^\e-\tilde g\|_{C^{1,\theta}(B_{r''}(x_0,0)\times I_{(r'')^2})}&\leq C\|u^\e-\tilde g\|_{W^{2,\frac{n+1}{1-\theta}}(B_{r''}(x_0,0)\times I_{(r'')^2})}\\
	&\leq O(\e)+O(\|\phi^\e\|_{C^{2,\theta}(B_{r'}(x_0,0)\times I_{(r')^2})}),
	\end{align*}
for any $\theta\in(0,1)$ and so
	\begin{align*}
	\|W'(u^\e)-W'(\tilde g)\|_{C^{1,\theta}(B_{r''}(x_0,0)\times I_{(r'')^2})}&\leq O(\e)+O(\|\phi^\e\|_{C^{2,\theta}(B_{r'}(x_0,0)\times I_{(r')^2})}),
	\end{align*}
since $W'$ is a polynomial.

Applying parabolic Schauder estimates of Lemma \ref{ParabolicSchauder} to the equation \eqref{eqn_DifferenceToFlatSolution} and bootstrapping then gives the estimate in the ball of radius $r$.
\end{proof}

\begin{proof}[Proof of Lemma \ref{LemmaErrorInZ}]
\eqref{CurvatureError} and \eqref{LaplacianError} are already proved in \cite[Section 8]{Wang2019a}, we only need to prove the remaining two estimates.

By our assumptions, we can assume without loss of generality that for sufficiently small $\e$, the level sets for values $(-1+b,1-b), b>0$ are graphs over $\{x_{n+1}=0\}$, that is $\{u^\e=s\}\cap \left\{D_{r'}\times I_{(r')^2}\right\}=\{x_{n+1}=h^{\e,s}(x_1,...,x_n,t)\}\cap \left\{D_{r'}\times I_{(r')^2}\right\}$ for $s\in(-1+b,1-b)$. We compute
	\begin{align*}
	&\frac{\partial u^\e}{\partial x_{n+1}}=\left(\frac{\partial h^{\e,s}}{\partial s}\right)^{-1},\\
	&\frac{\partial u^\e}{\partial x_i}=-\left(\frac{\partial h^{\e,s}}{\partial s}\right)^{-1}\frac{\partial h^{\e,s}}{\partial x_i},i=1,...,n,\\
	&\frac{\partial u^\e}{\partial t}=-\left(\frac{\partial h^{\e,s}}{\partial s}\right)^{-1}\frac{\partial h^{\e,s}}{\partial t}.
	\end{align*}
By a similar argument to Theorem \ref{GradientUpperLowerBound}, we have $|\frac{\partial u^\e}{\partial x_{n+1}}|>C^{-1}$ for some $C>0$ and thus $\frac{\partial h^{\e,s}}{\partial s}$ in the above computations makes sense.

For simplicity in the rest of the proof, we omit the superscripts $\e$ and $s$ in $u^\e$, $\phi^\e$ and $h^{\e,s}$ when there is no confusion.

By Lemma \ref{LemmaEstimateDifferenceToFlatSolution}, we see
	\begin{align*}
	\left|\frac{\partial h}{\partial t}\right|=O\left(\left|\frac{\partial u}{\partial t}\right|\right)=O(\e)+O(\|\phi^\e\|_{C^{2,\theta}(B_{r'}\times I_{(r')^2})}).
	\end{align*}
We further compute higher order derivatives of $h$ that contain one order of $t$ derivatives. By differentiating $u\left(x_1,...,x_n,h(x_1,...,x_n,t),t\right)=s$ twice, we have
	\begin{align*}
	\frac{\partial^2u}{\partial x_i\partial t}+\frac{\partial^2u}{\partial x_{n+1}\partial t}\frac{\partial h}{\partial x_i}+\frac{\partial^2h}{\partial x_i\partial t}\frac{\partial u}{\partial x_{n+1}}=0.
	\end{align*}
Since the term $\frac{\partial u}{\partial x_{n+1}}$ is uniformly bounded, and the first $2$ terms are of order $O(\e)+O(\|\phi^\e\|_{C^{2,\theta}(B_{r'}\times I_{(r')^2})})$, we see
	\begin{align*}
	\left|\frac{\partial^2 h^s}{\partial x_i\partial t}\right|=O(\e)+O(\|\phi^\e\|_{C^{2,\theta}(B_{r'}\times I_{(r')^2})}).
	\end{align*}
Similarly, to estimate $\left|\frac{\partial^3 h}{\partial x_i\partial x_j\partial t}\right|$, we differentiate again with respect to $x_j$ and get
	\begin{align*}
	&\frac{\partial^3u}{\partial x_i\partial x_j\partial t}+\frac{\partial^3u}{\partial x_{n+1}\partial x_j\partial t}\frac{\partial h}{\partial x_i}+\frac{\partial^2u}{\partial x_{n+1}\partial t}\frac{\partial^2 h}{\partial x_i\partial x_j}\\
	&+\frac{\partial^3h}{\partial x_i\partial x_j\partial t}\frac{\partial u}{\partial x_{n+1}}+\frac{\partial^2h}{\partial x_i\partial t}\frac{\partial^2 u}{\partial x_{n+1}\partial x_j} =0.
	\end{align*}
By the previous first and second order derivatives estimates (that contains $t$ derivatives), we see the first, second, third and fifth term are all of order $O(\e)+O(||\phi||_{C^{2,\theta}\left(B_{r'}\times I_{(r')^2}\right)})$. So from the fourth term, we have
	\begin{align*}
	\left|\frac{\partial^3h}{\partial x_i\partial x_j\partial t}\right|=O(\e)+O(\|\phi^\e\|_{C^{2,\theta}(B_{r'}\times I_{(r')^2})}).
	\end{align*}
Since the second fundamental form of the nodal sets $\Gamma^\e_t$ in graphical form is
	\begin{align*}
	A_{ij}(x_1,...,x_n,t)=-\frac{1}{\sqrt{1+|\nabla h^0|^2}}\frac{\partial}{\partial x_i}\left[\frac{1}{\sqrt{1+|\nabla h^0|^2}}\frac{\partial h^0}{\partial x_j}(x_1,...,x_n,t)\right].
	\end{align*}
Taking derivatives with respect to $t$ and using the above estimates of higher order derivatives of $h$ containing one order of $t$ derivatives, we get
	\begin{align*}
	\left|\frac{\partial }{\partial t}A\right|=O(\e)+O(\|\phi^\e\|_{C^{2,\theta}(B_{r'}\times I_{(r')^2})}).
	\end{align*}
And consequently, using \cite[(8.4)]{Wang2019a}, we get
	\begin{align*}
	\frac{\partial}{\partial t}\left(A_{\Gamma^\e_{z,t}}-A_{\Gamma^\e_{0,t}}\right)&=\frac{\partial}{\partial t}\left(zA_{\Gamma^\e_{z,t}}A_{\Gamma^\e_{0,t}}\right)\\
	&=zA_{\Gamma^\e_{z,t}}\frac{\partial}{\partial t}A_{\Gamma^\e_{0,t}}+zA_{\Gamma^\e_{0,t}}\frac{\partial}{\partial t}A_{\Gamma^\e_{z,t}}\\
	&=O(\e^2)+O(\|\phi^\e\|^2_{C^{2,\theta}(B_{r'}\times I_{(r')^2})}),
	\end{align*}
because $|A|,\left|\frac{\partial }{\partial t}A\right|=O(\e)$. This proves \eqref{CurvatureDerivativeError}.

The time derivative of the induced Riemannian metrics on the nodal sets can be estimated by
	\begin{align*}
	\left|\frac{\partial}{\partial t}\left[(g_{\Gamma^\e_{0,t}})_{ij}\right]\right|&=\left|\left[(A_{\Gamma^\e_{0,t}})_{ij}\right]\left\langle\frac{\partial X_{\Gamma^\e_{0,t}}}{\partial t},\nu_{\Gamma^\e_{0,t}}\right\rangle\right|\\
	&\leq O\left(\left|\left[(A_{\Gamma^\e_{0,t}})_{ij}\right]\frac{\partial h}{\partial t}\right|\right)\\
	&=O(\e^2)+O(\|\phi^\e\|^2_{C^{2,\theta}(B_{r'}\times I_{(r')^2})}).
	\end{align*}
And for the metrics on $\Gamma^\e_{z,t}$ for each $z$, we have by\cite[8.2]{Wang2019a}
	\begin{align*}
	(g_{\Gamma^\e_{z,t}})_{ij}=(g_{\Gamma^\e_{0,t}})_{ij}-2z\sum_{k=1}^n(A_{\Gamma^\e_{0,t}})_{ik}(g_{\Gamma^\e_{0,t}})_{jk}+z^2\sum_{k,l=1}^n(g_{\Gamma^\e_{0,t}})_{kl}(A_{\Gamma^\e_{0,t}})_{ik}(A_{\Gamma^\e_{0,t}})_{jl}.
	\end{align*}
Differentiating with respect to $t$ and using the $t$ derivative estimate of $A$ and $g_{\Gamma^\e_{0,t}}$ above, we then see
	\begin{align*}
	\left|\frac{\partial}{\partial t}\left[(g_{\Gamma^\e_{z,t}})_{ij}\right]\right|=O(\e)+O(\|\phi^\e\|_{C^{2,\theta}(B_{r'}\times I_{(r')^2})}).
	\end{align*}
Now by denoting $b^i_{\Gamma^\e_{z,t}}=\frac{1}{2}\sum_{j=1}^ng^{ij}_{\Gamma^\e_{z,t}}\frac{\partial\left(\log\det(g_{\Gamma^\e_{z,t}})_{ij}\right) }{\partial y_j}$, we can estimate the H\"older time norm of the error in $z$ of the Laplacian operator.
	\begin{align*}
	&\sup_{\substack{y\in B_r^n}}\left\|\Delta_{\Gamma^\e_{z,\cdot}}\varphi(y,\cdot)-\Delta_{\Gamma^\e_{0,\cdot}}\varphi(y,\cdot)\right\|_{C^\frac{\theta}{2}(I_{r^2})}\\
	&=\sup_{\substack{y\in B_r^n}}\left\|\sum_{i,j=1}^n\left(g^{ij}_{\Gamma^\e_{z,t}}-g^{ij}_{\Gamma^\e_{0,t}}\right)\frac{\partial^2\varphi}{\partial y_i\partial y_j}(y,\cdot)+\sum_{i=1}^n\left(b^i_{\Gamma^\e_{z,t}}-b^i_{\Gamma^\e_{0,t}}\right)\frac{\partial\varphi}{\partial y_i}(y,\cdot)\right\|_{C^\frac{\theta}{2}(I_{r^2})}\\
	&\leq \sup_{\substack{y\in B_r^n}}\left[O\left(\left|g^{ij}_{\Gamma^\e_{z,t}}-g^{ij}_{\Gamma_{\e,0,t}}\right|\right)\|\varphi\|_{C^{2,\theta}\left(B_r^n\times I_{r^2}\right)}+O\left(\|g_{\Gamma^\e_{z,\cdot}}\|_{C^{0,\frac{\theta}{2}}(I)}\right)\|\varphi\|_{C^2\left(B_r^n\times I_{r^2}\right)}\right]\\
	&\leq \sup_{\substack{y\in B_r^n}}O\left(\left|g^{ij}_{\Gamma^\e_{z,t}}-g^{ij}_{\Gamma^\e_{0,t}}\right|\right)\|\varphi\|_{C^{2,\theta}\left(B_r^n\times I_{r^2}\right)}+\sup_{\substack{y\in B_r^n\\|z| \leq 6|\log \e|}}O\left(\|g_{\Gamma^\e_{z,\cdot}}\|_{C^1(I)}\right)\|\varphi\|_{C^{2,\theta}\left(B_r^n\times I_{r^2}\right)}\\
	&= O(\e^2)+O(\|\phi^\e\|^2_{C^{2,\theta}(B_{r'}\times I_{(r')^2})})+O(\|\varphi\|^2_{C^{2,\theta}\left(B_r^n\times I_{r^2}\right)}),
	\end{align*}
and this proves \eqref{LaplacianHolderError}.
\end{proof}

\bibliography{AllenCahn.bib}

\newcommand{\etalchar}[1]{$^{#1}$}
\begin{thebibliography}{dPKW13}

\bibitem[Bra78]{brakke2015motion}
Kenneth~A. Brakke.
\newblock {\em The motion of a surface by its mean curvature}, volume~20 of
  {\em Mathematical Notes}.
\newblock Princeton University Press, Princeton, N.J., 1978.

\bibitem[CC06]{caffarelli2006phase}
Luis~A. Caffarelli and Antonio C\'{o}rdoba.
\newblock Phase transitions: uniform regularity of the intermediate layers.
\newblock {\em J. Reine Angew. Math.}, 593:209--235, 2006.

\bibitem[CM12]{colding2012generic}
Tobias~H. Colding and William~P. Minicozzi, II.
\newblock Generic mean curvature flow {I}: generic singularities.
\newblock {\em Ann. of Math. (2)}, 175(2):755--833, 2012.

\bibitem[CM20]{Chodosh2020}
Otis Chodosh and Christos Mantoulidis.
\newblock Minimal surfaces and the {A}llen-{C}ahn equation on 3-manifolds:
  index, multiplicity, and curvature estimates.
\newblock {\em Ann. of Math. (2)}, 191(1):213--328, 2020.

\bibitem[dPG18a]{Pino2018}
Manuel del Pino and Konstantinos~T. Gkikas.
\newblock Ancient multiple-layer solutions to the {A}llen-{C}ahn equation.
\newblock {\em Proc. Roy. Soc. Edinburgh Sect. A}, 148(6):1165--1199, 2018.

\bibitem[dPG18b]{Pino2018a}
Manuel del Pino and Konstantinos~T. Gkikas.
\newblock Ancient shrinking spherical interfaces in the {A}llen-{C}ahn flow.
\newblock {\em Ann. Inst. H. Poincar\'{e} Anal. Non Lin\'{e}aire},
  35(1):187--215, 2018.

\bibitem[dPKW11]{Pino2011}
Manuel del Pino, Michal Kowalczyk, and Juncheng Wei.
\newblock On {D}e {G}iorgi's conjecture in dimension {$N\geq 9$}.
\newblock {\em Ann. of Math. (2)}, 174(3):1485--1569, 2011.

\bibitem[dPKW13]{Pino2013}
Manuel del Pino, Michal Kowalczyk, and Juncheng Wei.
\newblock Entire solutions of the {A}llen-{C}ahn equation and complete embedded
  minimal surfaces of finite total curvature in {$\Bbb R^3$}.
\newblock {\em J. Differential Geom.}, 93(1):67--131, 2013.

\bibitem[Gua18]{Guaraco2018}
Marco A.~M. Guaraco.
\newblock Min-max for phase transitions and the existence of embedded minimal
  hypersurfaces.
\newblock {\em Journal of Differential Geometry}, 108(1):91--133, 2018.

\bibitem[Ilm93]{Ilmanen1993}
Tom Ilmanen.
\newblock Convergence of the {A}llen-{C}ahn equation to {B}rakke's motion by
  mean curvature.
\newblock {\em J. Differential Geom.}, 38(2):417--461, 1993.

\bibitem[Lie96]{lieberman1996second}
Gary~M. Lieberman.
\newblock {\em Second order parabolic differential equations}.
\newblock World Scientific Publishing Co., Inc., River Edge, NJ, 1996.

\bibitem[Man21]{Mantoulidis2021}
Christos Mantoulidis.
\newblock Allen-{C}ahn min-max on surfaces.
\newblock {\em Journal of Differential Geometry}, 117(1):93--135, 2021.

\bibitem[NW20]{Nguyen2020}
Huy~The Nguyen and Shengwen Wang.
\newblock Brakke regularity for the {A}llen-{C}ahn flow.
\newblock arXiv:2010.12378 [math.AP], 2020.

\bibitem[PQS07]{Polacik2007}
Peter Pol\'{a}\v{c}ik, Pavol Quittner, and Philippe Souplet.
\newblock Singularity and decay estimates in superlinear problems via
  {L}iouville-type theorems. {I}. {E}lliptic equations and systems.
\newblock {\em Duke Mathematical Journal}, 139(3):555--579, 2007.

\bibitem[Sat08]{Sato2008}
Norifumi Sato.
\newblock A simple proof of convergence of the {A}llen-{C}ahn equation to
  {B}rakke's motion by mean curvature.
\newblock {\em Indiana University Mathematics Journal}, 57(4):1743--1751, 2008.

\bibitem[Son97]{soner1997ginzburg}
Halil~Mete Soner.
\newblock Ginzburg-landau equation and motion by mean curvature, ii:
  Development of the initial interface.
\newblock {\em The Journal of Geometric Analysis}, 7(3):477--491, 1997.

\bibitem[Sun21]{sun2018entropy}
Ao~Sun.
\newblock On the entropy of parabolic {A}llen-{C}ahn equation.
\newblock {\em Interfaces Free Bound.}, 23(3):421--432, 2021.

\bibitem[T{\etalchar{+}}03]{tonegawa2003integrality}
Yoshihiro Tonegawa et~al.
\newblock Integrality of varifolds in the singular limit of reaction-diffusion
  equations.
\newblock {\em Hiroshima mathematical journal}, 33(3):323--341, 2003.

\bibitem[Tru08]{trumper2008relaxation}
Mariel~S{\'a}ez Trumper.
\newblock Relaxation of the curve shortening flow via the parabolic
  {G}inzburg-{L}andau equation.
\newblock {\em Calculus of Variations and Partial Differential Equations},
  31(3):359--386, 2008.

\bibitem[Wan03]{Wang2003}
Li~He Wang.
\newblock A geometric approach to the {C}alder\'{o}n-{Z}ygmund estimates.
\newblock {\em Acta Mathematica Sinica (English Series)}, 19(2):381--396, 2003.

\bibitem[Wan17]{wang2014new}
Kelei Wang.
\newblock A new proof of {S}avin's theorem on {A}llen-{C}ahn equations.
\newblock {\em J. Eur. Math. Soc. (JEMS)}, 19(10):2997--3051, 2017.

\bibitem[Whi05]{white2005local}
Brian White.
\newblock A local regularity theorem for mean curvature flow.
\newblock {\em Ann. of Math. (2)}, 161(3):1487--1519, 2005.

\bibitem[WW19a]{Wang2019a}
Kelei Wang and Juncheng Wei.
\newblock Finite {M}orse index implies finite ends.
\newblock {\em Comm. Pure Appl. Math.}, 72(5):1044--1119, 2019.

\bibitem[WW19b]{Wang2019}
Kelei Wang and Juncheng Wei.
\newblock Second order estimate on transition layers.
\newblock {\em Adv. Math.}, 358:106856, 2019.

\end{thebibliography}
\bibliographystyle{alpha}

\end{document}